\newif\ifmakepreprint
\makepreprinttrue %Non-PREPRINT VERSION
\ifmakepreprint
	\documentclass[onefignum,onetabnum]{siamart251104}
\else
	\documentclass[review,onefignum,onetabnum]{siamart251104}
\fi
\usepackage{accents}
\usepackage[toc,page]{appendix}
\usepackage{enumerate}
\usepackage{amssymb}

\usepackage{graphicx}
\usepackage{subcaption}
\usepackage[margin=1in]{geometry}

\newsiamthm{lmm}{Lemma}
\newsiamthm{thrm}{Theorem}
\newsiamthm{rmrk}{remark}
\newsiamthm{dfntn}{Definition}
\newsiamthm{system}{System}
\newsiamthm{assumption}{Assumption}
\newsiamthm{crllr}{Corollary}

\newcommand{\bv}[1]{{\mathbf #1}} 
\newcommand{\bt}[1]{{\mathbf{#1}}}
\newcommand{\myrank}[1]{{\texttt{rank}\left( {#1}\right)}} 

\usepackage{algorithm}
\usepackage{algpseudocode}
\usepackage{chngcntr}
\counterwithout{algorithm}{section} %% Use simple numbers for algorithm (instead of Alogithm *SectionNumber*.*LocalAlgoNumber*

\newcommand{\funSpace}[1]{\mathcal{#1}} 
    \newcommand{\fom}[1]{{{#1}}} %model-order-reduced
    \newcommand{\com}[1]{{{#1}_{\rm{c}}}} %complexity-reduced
    
        \newcommand{\prol}[1]{{{#1}_{\rm{p}}}} %prolongated compressed
        \newcommand{\thi}[1]{{{#1}_{\rm{t}}}}	% thin compressed
    
\newcommand{\rom}[1]{{{#1}_{\rm{r}}}}    
\newcommand{\Nr}{{N_r}}

\newcommand{\R}{\mathbb{R}}

\newcommand{\ffun}{{\bv{f}}}

\newcommand{\bform}{{b}}
\newcommand{\tL}{{\fs{\varphi}}}

\newcommand{\tAlg}{{\bv{p}}}

\newcommand{\fs}[1]{{{#1}}} %% function in function space
\newcommand{\wei}{{\omega}} % quadrature weight
\newcommand{\xp}{{\xi}} % spatial point (x)

\newcommand{\acro}[1]{{\texttt{#1}}}  %%%acronyms
\newcommand{\mydim}[1]{{\textrm{dim}#1}}  %%%acronyms

\newcommand{\vwei}{\bv{u}} % vector quadrature weight
\newcommand{\Ap}{{\bt{A}}}
\newcommand{\vAp}{\bv{a}}

\newcommand{\fopt}{{\eta}}

\newcommand{\dreg}{{{\bv{d}}}}

\newcommand{\Cp}{{\bt{C}}}
   
\newcommand{\xrf}{{\rom{\bar{x}}}}

\newcommand{\mydiag}{\rm{diag}}    
	\newcommand{\blsRevB}[1]{{{#1}}} %revision 2
\newcommand{\xrfVec}{{\nabla \rom{\bv{x}}}}       % vector-valued ROM state, \xrfVec: Omega -> R^d
\newcommand{\phiVec}{{\rom{\vec{\varphi}\mkern-2mu}}} % vector-valued ROM test function, \phiVec: Omega -> R^d
\newcommand{\nlField}{{\bt{S}}}                    % matrix-valued nonlinear field, \nlField: V x Omega -> R^{d x d}
\newcommand{\geoTrans}{{\boldsymbol{\Psi}}}        % geometric transformation factor (Section 5.4)

\usepackage{tikz}
\usetikzlibrary{calc,positioning,shapes,arrows,decorations.shapes} 
\tikzstyle{rect}=[
   rectangle, draw=black, font=\small\sffamily\bfseries, inner sep=9pt]

\newtheorem{prblm}{Problem} %TODO-maybe adapt number [theorem]

\begin{document}

\title{Reducing Training Complexity in Empirical Quadrature-Based Model Reduction via Structured Compression}

%\funding{
%%The support of --- is acknowledged.}}
%}

\author{Bj\"orn Liljegren-Sailer\thanks{Radon Institute for Computational and Applied Mathematics (RICAM), Austrian Academy
of Sciences, Altenberger Str. 69, 4040 Linz, Austria; and MathConsult GmbH, Altenberger Str. 69, 4040 Linz, Austria
  (\email{Corresponding author: bsailer@ricam.oeaw.ac.at}).}
%%\and (NAME XYZ)\footnotemark[2]
}
\maketitle
\date{\today}
\newcommand{\blfootnote}[1]{%
  \begingroup
  \renewcommand{\thefootnote}{}\footnote{#1}%
  \endgroup
}
\blfootnote{Preprint accepted for publication in the SIAM Journal on Scientific Computing (SISC).}

\begin{abstract}
Model order reduction seeks to approximate large-scale dynamical systems by lower-dimensional reduced models. For linear systems, a small reduced dimension directly translates into low computational cost, ensuring online efficiency. This property does not generally hold for nonlinear systems, where an additional approximation of nonlinear terms --known as complexity reduction-- is required. To achieve online efficiency, empirical quadrature and cell-based empirical cubature are among the most effective complexity reduction techniques. However, the offline training of these methods operates on a matrix whose dimension scales with both the snapshot count and the reduced model dimension, and can become a computational bottleneck at larger scale. Existing strategies such as parallelization and randomized linear algebra reduce the cost of processing this matrix but do not reduce its dimension.
In this paper, we introduce a preprocessing approach based on a specific structured compression of the training data. Crucially, our approach ensures that no operation scales concurrently with the snapshot count, the reduced model dimension, and the problem dimension. Overall, this yields roughly an order-of-magnitude reduction in offline computational cost and memory requirements, thereby enabling the application of the complexity reduction methods to larger-scale problems. Accuracy is preserved, as indicated by our error analysis and demonstrated through numerical examples.
\end{abstract}

\begin{AMS}
65D32, 65F55
\end{AMS}
%% 65D32 (Quadrature/Cubature) + 65F55 (Low-rank/POD compression) 

%
\begin{keywords}
	model reduction; nonlinear model reduction; hyper-reduction; complexity reduction; structure-preserving; empirical quadrature; empirical cubature
\end{keywords}

%%\section*{Introduction}
Let a  \textit{truth} nonlinearity $\ffun: \R^N \to \R^N$ of high dimension $N$ be given. We define its projected counterpart by
\begin{align*}
	\rom{\ffun}: \R^\Nr \to \R^\Nr, \hspace{1cm} \rom{\ffun}: \rom{\bv{x}} \mapsto \bt W^T \ffun(\bt V \rom{\bv{x}}),
\end{align*}
where $\bt W, \bt V \in \R^{N,\Nr}$  induce a projection onto a low dimension $\Nr \ll N$.
Although the nonlinearity $\rom{\ffun}$ is low-dimensional its evaluation cost is --in the general nonlinear case-- not independent of the truth dimension $N$. This is a typical scenario, specifically when nonlinear partial differential equations are approximated using projection-based model order reduction methods. In order to avoid the high evaluation costs during time-critical online computations, a complexity-reduction (also known as hyper-reduction) is performed. In an offline training phase a complexity-reduced nonlinearity
\begin{align*}
\com{\ffun}: \R^\Nr \to \R^\Nr, \hspace{1cm } \text{with} \hspace{0.3cm} \com{\ffun} \approx \rom{\ffun}
\end{align*}
is constructed as a surrogate for $\rom{\ffun}$. Online efficiency is achieved by imposing some sort of sparsity on the approximation, which avoids the evaluation costs to scale with the full dimension $N$.
%%%%
One can distinguish two classes of complexity-reduction methods.

In this work, we focus on those that adopt the \textit{project-then-approximate} paradigm, and our goal is to enhance the offline training algorithms associated with these methods.
The empirical quadrature \cite{art:equad-eff-yano,art:hyperreduction-fem-cubature,art:lp-cub-yano} and the cell-based empirical cubature \cite{art:comred-ecsw, art:cub-efficient-integration-cubature, inbook:ecsw-compBottle} are the principal representatives of project-then-approximate techniques. As the name suggests, the training phase of these methods seeks to approximate the projected nonlinearity directly. This approach typically yields a complexity reduction with high fidelity and robust online performance. Moreover, an additional advantage of this class of methods is that fundamental structural properties --such as dissipative relations or a conservation law interpretation-- can be preserved through simple variational arguments in many relevant settings.

The second class of complexity-reduction methods, referred to as \textit{approx{\-}imate{\-}-then{\-}-project} methods, generally exhibit less robust online performance in the sense of physical structure preservation, particularly in applications where the aforementioned structural properties play a critical role~\cite{art:morParamHamHesthaven,art:lilsailer-nlfow,art-equad-DG-chan}. Nonetheless, they seem to be applied more in practice, specifically the Empirical Interpolation Method and its discrete variant~\cite{art:deim-state-space-err,art:deim-canon-str}.
While empirical interpolation methods can be applied to a broad class of \acro{FOM}s to gain online speed-up, empirical quadrature approaches are most naturally employed in the context of finite element or finite volume discretizations.
The primary advantage of the empirical interpolation methods lies in the availability of highly efficient offline training procedures that operate exclusively on compressed training data.

Also for empirical quadrature methods, several strategies have been proposed to make this training feasible at larger scale, including massive parallelization~\cite{art:equad-eff-yano,art:grimberg2021mesh} and compression via randomized or partitioned \acro{SVD}~\cite{hernandez2024cecm,art:hyperreduction-fem-cubature,art:MartinssonVoronin2016}. These approaches reduce the cost of processing the training matrix, but operate on a matrix whose row dimension inherently scales with $K\Nr$, i.e., with both the snapshot count and the reduced model dimension. To the best of the author's knowledge, no approach that reduces this dimension directly has been proposed for general project-then-approximate methods---with the notable exception of \cite{art:Guo2024,art:Guo2024-secOrd}, which exploits a particular data separation available in a continuum mechanics setting; cf.\ Section~\ref{subsec:separatedData-tensorval}. In this work, we introduce a preprocessing that leverages the structure inherent in the training problems to reduce the training-matrix row dimension directly and applies in a general setting. Overall, our approach improves scalability and can be combined with the strategies above to address even larger-scale problems.

\subsection*{Organization of the paper} %TODO 
The structure of the paper is as follows. The spatial discretization and the model order reduction setting for this paper is established in Section~\ref{sec:motivSetting}. The project-then-approximate complexity reduction approach is described in Section~\ref{sec:ansatzEquadCub}. Section~\ref{sec:offline-training-standard} reviews existing implementations for the training of the complexity reduction. Our proposed preprocessing approach based on a structured compression is then introduced. The general data compression strategy is derived in Section~\ref{sec:preprocessing-ansatz}, while Section~\ref{sec:separatedData} develops related structured representations tailored to our specific training problems. Section~\ref{sec:error-analysis} provides an error analysis, including an a posteriori bound. Finally, Section~\ref{sec:numericPerformance} presents numerical performance tests demonstrating the efficiency of our proposed data compression approach and the reliability of the a posteriori bound.

\subsection*{Notation}
Matrices, vectors, and scalars are indicated by capital boldfaced, small boldfaced, and normal letters, respectively, whereby vector expressions
always refer to column vectors. The norm $\|.\|$ refers to the spectral norm for matrices and to the Euclidean norm for vectors, respectively. 
We write $\|.\|_{\rm{F}}$ for the Frobenius norm, i.e., $\| \bt{Q} \|_{\rm{F}}^2 = { \sum_{i=1}^M \sum_j^N  |q_{ij}|^2}$ for any $\bt{Q}\in \R^{M,N}$.
The symbol '$\circ$' denotes the the Hadamard product, which is the elementwise multiplication operator between two vectors of the same dimension.
% i.e., $\circ: \R^N \times \R^N \to \R^N$.
Given a vector $\bv{v}\in \R^N$, we write $\mydiag(\bv{v}) \in \R^{N,N}$ for the diagonal matrix with diagonal entries given by the entries of $\bv{v}$, and $[\bv{v}]_i$ for its $i$-th entry of the vector. Moreover, the cardinality of any finite set $I_c$ is denoted by $|I_c|$.

\section{Setting for the paper}\label{sec:motivSetting}

\subsection{Model hierarchy  in the algebraic setting}\label{subsec:intro-hierarchy}

Consider an input-output map $\bv{u} \mapsto \fom{\bv{y}}$ described in terms of the dynamical system
\begin{align*}
\frac{d}{dt} \fom{\bv{x}}(t) =  \ffun(\fom{\bv{x}}(t)) + \bt{A} \bv{x}(t) + \fom{\bt{B}} \bv{u}(t) , \qquad \fom{\bv{y}}(t) = \fom{\bt{C}} \fom{\bv{x}}(t)
\end{align*}
for time $t \in [0,T]$ and system matrices $\bt{A}\in \R^{N,N}$, $\fom{\bt{B}}\in \R^{N,p}$, $\fom{\bt{C}}\in \R^{q,N}$. We assume  $p,q \ll N$, which motivates the use of model reduction. The $\ffun: \R^N \to \R^N$ describes the truth nonlinearity and the latter system is denoted as the full order model (\acro{FOM}), which can be closed by appropriate initial conditions.

In the nonlinear case, the model reduction consists of two steps. Firstly, the model order reduction is performed, using reduction bases $\bt W, \bt V \in \R^{N,\Nr}$ and $\Nr \ll N$. Assuming the usual bi-orthogonality assumption $\bt{W}^T\bt{V}= \bt{I}$ holds, the reduced order model (\acro{ROM}) is then given as the following projection of the \acro{FOM},
\begin{align*}
\frac{d}{dt} \rom{\bv{x}}(t) =   \rom{\ffun}(\rom{\bv{x}}(t)) + \rom{\bt{A}} \rom{\bv{x}}(t) + \rom{\bt{B}} \bv{u}(t) , \qquad \rom{\bv{y}}(t) = \rom{\bt{C}} \rom{\bv{x}}(t),
\end{align*}
with reduced matrices $\rom{\bt{A}}= \bt{W}^T\bt{A} \bt{V}$,  $\rom{\bt{B}} = \bt{W}^T\bt{B}$, $\rom{\bt{C}} = \bt{C} \bt{V}$.
The second step, termed complexity reduction, involves approximating the projected nonlineartity $\rom{\ffun}: \rom{\bv{x}} \mapsto \bt W^T \ffun(\bt V \rom{\bv{x}})$ by a complexity-reduced nonlinearity $\com{\ffun}: \R^\Nr \to \R^\Nr$. Replacing the projected nonlinearity in the \acro{ROM} with $\com{\ffun}$ yields the complexity-reduced model (\acro{CROM}). In this paper, we focus primarily on the complexity reduction step, and, more specifically, on its training phase.

\subsection{Projected nonlinearity in the function space setting}\label{subsec:intro-nonlinearity}
The underlying function space perspective is also relevant for our approach. We assume that the \acro{FOM} arises from the discretization of a partial differential equation posed on a bounded spatial domain $\Omega \subset \R^d$, with $d \in \{1,2,3\}$. For simplicity, we restrict ourselves to a standard finite element setting and assume that $\ffun$ in the \acro{FOM} corresponds to a scalar-valued function $f: \Omega \to \R$.

Let $\funSpace{W}$ and $\funSpace{V}$ denote the test and trial spaces of the finite element discretization, respectively, both of dimension $N$. The basis of $\funSpace{W}$ is given by $\tL^1,\ldots,\tL^N: \Omega \to \R$.
We denote the \acro{FOM} state by $\fs{x}: [0,T] \to \funSpace{V}$ for $t \in [0,T]$, and assume that the truth nonlinearity ${\bt{f}}: \R^N \to \R^N$ can be written as
\begin{align*}
[\ffun(\bv{x}(t))]_n = \int_\Omega f(\fs{x}(t)) \, \tL^n \, d\xp, \qquad \text{for } n = 1,\ldots,N,
\end{align*}
where $\bv{x}$ denotes the coordinate representation of $\fs{x}$ in the \acro{FOM} basis.
The \acro{ROM} is obtained by projecting the \acro{FOM} onto reduced spaces $\rom{\funSpace{W}} \subset \funSpace{W}$ and $\rom{\funSpace{V}} \subset \funSpace{V}$, with $\mydim(\rom{\funSpace{W}}) = \mydim(\rom{\funSpace{V}}) = \Nr$. Let the \acro{ROM} state $\rom{\fs{x}}: [0,T] \to \rom{\funSpace{V}}$ approximate $\fs{x}$. The \acro{ROM} nonlinearity $\rom{\bt{f}}: \R^\Nr \to \R^\Nr$ is then defined by
\begin{align*}
[\rom{\ffun}(\rom{\bv{x}}(t))]_n = \int_\Omega f(\rom{\fs{x}}(t)) \, \rom{\tL}^n \, d\xp, \qquad \text{for } n = 1,\ldots,\Nr,
\end{align*}
where $\rom{\tL}^1,\ldots,\rom{\tL}^\Nr$ form a basis of $\rom{\funSpace{W}}$, and $\rom{\bv{x}}$ is the coordinate representation of $\rom{\fs{x}}$ in the \acro{ROM} basis.
\begin{rmrk} \label{rem:coordRepresDyn}
The continuous setting described here relates to the algebraic framework in Section~\ref{subsec:intro-hierarchy} as follows: the columns of the reduction basis $\bt{W}$ are the coordinate representations of the test functions $\rom{\tL}^1,\dots,\rom{\tL}^\Nr$ in the \acro{FOM} space $\funSpace{W}$.  Likewise, the columns of $\bt{V}$ represent the coordinates of $\rom{\funSpace{V}}$.
%%Similarly, the column span of $\bt{V}$ represents the coordinates of $\rom{\funSpace{V}} \subset \funSpace{V}$, chosen to ensure bi-orthogonality between $\bt{V}$ and $\bt{W}$.
\end{rmrk}

\section{Complexity reduction ansatz} \label{sec:ansatzEquadCub}
A complexity reduction method constructs a cheaper-{\-}to-{\-}evaluate surrogate for the projected nonlinearity $\rom{\ffun}$. In the case of the project-then-approximate methods this is done based on a \textit{truth sum representation} of the nonlinearity. We shortly introduce the empirical quadrature and the empirical cubature setting in Sections~\ref{subsec:empQuad}-\ref{subsec:empCub}. As the two approaches share many similarities, we then go over to a more generic setting in Section~\ref{sec:genericSetting} to formalize the respective training problem in the algebraic setting.

\subsection{Empirical quadrature} \label{subsec:empQuad}
In empirical quadrature, the truth representation is defined by a general quadrature rule with points ${\xp}_m$ and weights $\tilde{\wei}_m$ for $m = 1,\ldots,M_{\rm quad}$. We assume that the \acro{FOM} nonlinearity is evaluated using this truth quadrature, i.e.,
\begin{align*}
[\ffun(\bv{x}(t))]_n = \tilde{\bform}\!^{\,\rm{quad}}(f(\fs{x}(t)),\tL^n) := \sum_{m=1}^{M_{\rm{quad}}} \tilde{\wei}_m {f}(\fs{x}(t,{\xp}_m))\tL^n({\xp}_m)
%%, \hspace{0.5cm} \text{for } n = 1,\ldots, N.
\end{align*}
for $n = 1,\ldots, N$.
By construction the $\tilde{\bform}\!^{\,\rm{quad}}$ is a bilinear form. Here and in the following, we use the same symbol for the state $\fs{x} : [0,T] \to \mathcal{V}$ and its real-valued counterpart $\fs{x} : [0,T] \times \Omega \to \R^d$, to keep the notation concise. The projected nonlinearity allows for a similar representation, given as
\begin{align*}
[\rom{\ffun}(\rom{\bv{x}}(t))]_n = \tilde{\bform}^{\rm{quad}}({f}(\rom{\fs{x}}(t)), \rom{\tL}\!^n) =  \sum_{m=1}^{M_{\rm{quad}}} \tilde{\wei}_m {f}(\rom{\fs{x}}(t,{\xp}_m))\rom{\tL}\!^n({\xp}_m)
\end{align*}
for  $n = 1,\ldots, \Nr$. It is important to note that the truth nonlinearity typically exhibits sparsity, meaning that most summands in the sum representation of $[\ffun(\cdot)]_n$ are zero.
However, no sparsity can generally be expected for $\rom{\ffun}$, so its evaluation typically scales with $M_{\rm quad} \geq N$.
%%As a consequence, the computational complexity of $\rom{\ffun}$ is about the same as for $\ffun$.

%%
As a remedy, empirical quadrature can be used. The idea of the latter is  to replace the truth bilinear form $\tilde{\bform}^{\rm{quad}}$ by a complexity-reduced version $\com{\tilde{\bform}}\!^{\rm{quad}}$.
The resulting complexity-reduced nonlinearity $\com{\ffun}$ then reads
\begin{align*}
[\com{\ffun}(\rom{\bv{x}}(t))]_n = \com{\tilde{\bform}}\!^{\rm{quad}}({f}(\rom{\fs{x}}(t)), \rom{\tL}\!^n) :=  \sum_{m \in I_c} {\wei}_m {f}(\rom{\fs{x}}(t,{\xp}_m) \rom{\tL}\!^n({\xp}_m)
%, \hspace{1cm} \text{for } n= 1,\ldots \Nr.
\end{align*}
for $n= 1,\ldots \Nr$.
The subset $I_c \subset \{1,\ldots, {M_{\rm quad}} \}$ and weights $\wei_m \geq 0$, $m\in I_c$, are subject to an offline training. Following \cite{art:hyperreduction-fem-cubature}, we consider the following training problem:
Given snapshots $\fs{x}^1, \ldots, \fs{x}^K \subset \funSpace{V}$ of the \acro{FOM} solution and $\com{M} \ll M_{\rm{quad}}$, and choosing a regularization (reg.), we solve
{\small
\begin{align*} %%\label{eq:opt-quad}
	\min_{\substack{I_c, \,|I_c| \leq \com{M}\\
			            \wei_m \geq 0, \,m \in I_c } }
\sum_{\substack{k = 1,\ldots,K\\ n= 1,\ldots,\Nr}}			            
			             \left|\left| \left(\sum_{m \in I_c} {\wei}_m {f}({\fs{x}}^k({\xp}_m))\rom{\tL}\!^n({\xp}_m) \right)
			             - \tilde{\bform}^{\rm{quad}}({f}({\fs{x}}^k), \rom{\tL}\!^n)\right|\right|^2
\quad \text{(+reg.)}
\end{align*}
}%%
%%% 
\begin{rmrk}
Note that snapshots ${\fs{x}}^k \in \funSpace{V}$ of the \acro{FOM} are used in the training problem. It might seem more natural to use snapshots restricted to the reduced space $\rom{\funSpace{V}}$ instead. However, the evaluation with projected training data requires additional computational costs with very minor impact on the result of the training, cf.~\cite{art:hyperreduction-fem-cubature,art:comred-ecsw}. The same can be said and done for other project-then-approximate complexity reduction methods.
 \end{rmrk}

%%%%
\subsection{Cell-based empirical cubature} \label{subsec:empCub}
In the standard finite element setting, the domain $\Omega$ is partitioned into non-overlapping cells $\Omega_1,\ldots,\Omega_{M_{\rm cells}}$. The (cell-based) empirical cubature approach exploits a sum representation of $\ffun$ induced by this partitioning. Thus, it holds
\begin{align*}
[\ffun(\bv{x}(t))]_n = \tilde{\bform}^{\rm cells}(f(\fs{x}(t)),\tL^n) := \sum_{m=1}^{M_{\rm cells}} \int_{\Omega_m} f(\fs{x}(t)) \, \tL^n \, d\xp,
\end{align*}
for $n = 1,\ldots,N$. Most summands vanish due to the small local support of the test functions.
The projected nonlinearity admits a similar representation, given by
\begin{align*}
[\rom{\ffun}(\rom{\bv{x}}(t))]_n = \tilde{\bform}^{\rm cells}(f(\rom{\fs{x}}(t)),\rom{\tL}^n) = \sum_{m=1}^{M_{\rm cells}} \int_{\Omega_m} f(\rom{\fs{x}}(t)) \, \rom{\tL}^n \, d\xp,
\end{align*}
for $n = 1,\ldots,\Nr$. Here, generally no summands vanish, so the evaluation cost scales with $M_{\rm cells}$ rather than $\Nr$.

Complexity reduction is achieved by replacing $\tilde{\bform}^{\rm cells}$ with a sparse bilinear form ${\tilde{\bform}}_c\!^{\rm cells}$. The resulting approximation $\com{\ffun}$ then reads
\begin{align*}
[\com{\ffun}(\rom{\bv{x}}(t))]_n = \com{\tilde{\bform}}\!^{\rm{cells}}({f}(\rom{\fs{x}}(t)), \rom{\tL}\!^n) :=
\sum_{m\in I_c} \wei_m \int\limits_{\Omega_m} {f}(\rom{\fs{x}}(t)) \rom{\tL}\!^n \,d\xp,
\end{align*}
for $n = 1,\ldots, \Nr$.
The subset $I_c \subset \{1,\ldots, M_{\rm cells} \}$ together with the weights $\wei_m \geq 0$, $m \in I_c$, are determined by the following training problem; cf.~\cite{art:comred-ecsw}:

Given snapshots $\fs{x}^1, \ldots, \fs{x}^K \subset \funSpace{V}$ of the \acro{FOM} solution and $\com{M} \ll M_{\rm cells}$ and choosing regularization (reg), we solve
{\small
\begin{align*} %%\label{eq:opt-cub}
	\min_{\substack{I_c, \,|I_c| \leq \com{M}\\
			            \wei_m \geq 0, \,m \in I_c } }
\sum_{\substack{k = 1,\ldots,K\\ n= 1,\ldots,\Nr}}			            
			             \left|\left| \left(\sum_{m\in I_c} \wei_m \int\limits_{\Omega_m} {f}({\fs{x}}^k) \rom{\tL}\!^n \,d\xp \right)
			             - \tilde{\bform}^{\rm{cells}}({f}({\fs{x}}^k), \rom{\tL}\!^n)\right|\right|^2
\quad \text{(+reg.)}
\end{align*}
}%%endsmall

\begin{rmrk}
For the finite-element settings considered in this paper, the cell-based approach is mathematically equivalent to the energy-con{\-}serving sampling and weighting approach (ECSW)~\cite{art:comred-ecsw,art:chan-EntropROM}, though ECSW is typically formulated in terms of element-level discrete operators.
\end{rmrk}

\subsection{Generic and algebraic setting} \label{sec:genericSetting}
In slight generalization to the empirical quadrature and cubature setting, let us assume that a truth sum representation of the projected nonlinearity $\rom{\ffun}$ is given by
\begin{align*}
[\rom{\ffun}(\xrf)]_n= \tilde{\bform}\!^{\,\rm{*}}(f(\xrf),\rom{\tL}\!^n) := \sum_{m=1}^{M_{}} \tilde{\wei}_m  \beta_{\rm{*}}^m(f(\xrf),\rom{\tL}\!^n), \hspace{0.7cm} \text{for } n = 1,\ldots, \Nr
\end{align*}
and $\xrf \in \rom{\funSpace{V}}$. We additionally pose the assumption that each of the localized bilinear terms $\beta_{\rm{*}}^m$ has a low computational cost. This then motivates a complexity reduction based on an approximation of $\tilde{\bform}\!^{\,\rm{*}}$ with a sparse sum. The respective complexity-reduced nonlinearity $\com{\ffun}$ becomes
\begin{align*}
[\com{\ffun}(\xrf)]_n ={\tilde{\bform}}\!^{\,\rm{*}}_c(f(\xrf),\rom{\tL}\!^n) 
:= \sum_{m\in \com{I}} {\wei}_m  \beta_{\rm{*}}^m(f(\xrf),\rom{\tL}\!^n), \hspace{0.7cm} \text{for } n = 1,\ldots, \Nr.
\end{align*}
The subset $I_c \subset \{1,\ldots , M  \}$ and weights $\wei_m \geq 0$ of the complexity reduction are subject to the following training problem:

Given snapshots $\fs{x}^1, \ldots, \fs{x}^K \subset \funSpace{V}$ and $\com{M}\ll M_{}$, solve
{\small
\begin{align} 
\begin{split}\label{eq:opt-cubquad}
	\min_{\substack{I_c, \,|I_c| \leq \com{M}\\
			            \wei_m \geq 0, \,m \in I_c } }
\sum_{\substack{k = 1,\ldots,K\\ n= 1,\ldots,\Nr}}			            
			             \left|\left| \left(\sum_{m \in I_c} {\wei}_m \beta_{\rm{*}}^m({f}({\fs{x}}^k),\rom{\tL}\!^n) \right)
			             - \tilde{\bform}\!^{\,\rm{*}}({f}({\fs{x}}^k), \rom{\tL}\!^n)\right|\right|^2 \\
			             + \left( \left(\sum_{m \in I_c} {\wei}_m \beta_{\rm{*}}^m(1,1) \right) - \tilde{\bform\!^{\,\rm{*}}}(1,1) \right)^2 .
\end{split}
\end{align}
}%%
This formulation generalizes the training problems defined above for empirical quadrature and cubature, but the regularization term is now fixed to penalize the approximation error in reconstructing $\tilde{\bform}^{\rm *}(1,1)$. As noted in \cite{art:hyperreduction-fem-cubature}, this regularization prevents a specific type of ill-posedness associated with trivial solutions. Its presence will also play a role in our error analysis (cf.~\Cref{the:CompressedOptFun}). The training problem \eqref{eq:opt-cubquad} represents a non-negative least-squares program with cardinality constraints \cite{art:nguyen2019non,proc:blumensath2008gradient}. In the following, we rewrite this problem in an algebraic setting.

Let the data vectors $\vAp^{k,n}$ for $k = 1,\ldots,K$ and $n = 1,\ldots,\Nr$, and the solution manifold matrix $\tilde{\Ap}$ be defined by
{\small
\begin{align} \label{eq:Ab-data}
	\vAp^{k,n} = 
\begin{bmatrix}
\beta_{\rm{*}}^1\left( f(\bv{x}^k), \rom{\tL}\!^n \right) \\
\vdots \\
\beta_{\rm{*}}^M\left( f(\bv{x}^k), \rom{\tL}\!^n \right)
\end{bmatrix} \in \R^{M},
\hspace{0.4cm}
\text{ and } \hspace{0.4cm}
	 	\tilde{\Ap} = \begin{pmatrix}
		(\vAp^{1,1})^T \\
		\vdots \\
		(\vAp^{1,\Nr})^T \\
		(\vAp^{2,1})^T \\
		\vdots \\
		(\vAp^{2,\Nr})^T \\
		\vdots \\
		(\vAp^{K,1})^T \\
		\vdots \\
		(\vAp^{K,\Nr})^T
	\end{pmatrix}  \in \R^{K\Nr , M} .
\end{align}
}%%
Further, define the truth weight vector $\tilde{\vwei}=[\tilde{\wei}_1,\ldots,\tilde{\wei}_M]^T$ and the vector $\dreg =[\beta_{\rm{*}}^1(1,1),\ldots,\beta_{\rm{*}}^M(1,1)]^T$ related to the regularization term. Then \eqref{eq:opt-cubquad} is equivalent to the following problem.
%%
%%An equivalent algebraic formulation of \eqref{eq:opt-cubquad}be stated.

\begin{prblm}\label{prob:unc-CrTrain}
%%\text{ $ $}\\[-1.5cm]
\begin{align*}
\text{Solve} \quad & \min_{\substack{I_c \subset \{1,\ldots,M\},\, |I_c| \leq \com{M} \\ \vwei=[\wei_1,\ldots,\wei_M]^T}} 
\| \tilde{\Ap}(\vwei - \tilde{\vwei}) \|^2 + (\dreg^T(\vwei - \tilde{\vwei}))^2 \\
& \text{s.t.} \quad \wei_m \geq 0 \text{ for } m \in I_c,\quad \wei_{\bar{m}} = 0 \text{ for } \bar{m} \notin I_c.
\end{align*}
\end{prblm}
\begin{rmrk}\label{rem:standard-least-squares}
Note that the cost function of \Cref{prob:unc-CrTrain} can be rewritten in the standard least-square form $F: \vwei \mapsto \| \mathcal{A} \vwei -\bv{g} \|^2$, defining  $\mathcal{A} =  [\tilde{\Ap}^T, \dreg]^T$ and  $\bv{g} = \mathcal{A} \tilde{\vwei}$. 
\end{rmrk}
%%Thus, we have a non-negative least squares problem with a sparsity constraint, which we supplemented by a regularization term.

\section{Greedy offline training from literature} \label{sec:offline-training-standard}
The training phase of complexity-reduction methods inherently involves a combinatorial component that scales with the problem dimension. In \Cref{prob:unc-CrTrain}, this corresponds to selecting the subset $I_c$. Solving this problem to global optimality is generally infeasible and unnecessary. Consequently, greedy strategies are employed in practice. We adopt the Orthogonal Matching Pursuit \cite{art:nguyen2019non,proc:blumensath2008gradient}, which is well-suited to the structure of our training problem. The core idea is to alternate between expanding the index set $I_c$ through a greedy search and computing optimal weights for the fixed set. A possible implementation is sketched in \Cref{alg:OrthMatchPur}.

\begin{algorithm}[H]
\caption{Orthogonal Matching Pursuit for offline training}
\label{alg:OrthMatchPur}
\begin{algorithmic}[1]
\State \textbf{Input:} Training data $(\mathcal{A}, \bv{g})$, dimension set $\com{M}$
\State \textbf{Output:} Sparse weights $\vwei^{\com{M}}$, residual $\sqrt{F(\vwei^{\com{M}})}$
\State Define $F(\vwei) = \|\mathcal{A}\vwei - \bv{g}\|^2$, with gradient $\nabla F\blsRevB{(\vwei)} = 2\mathcal{A}^T(\mathcal{A}\vwei-\bv g)$
\State Initialize $I_0 = \emptyset$, $\vwei^0 = \bv{0}$
\For{$k = 1:\com{M}$}
    \State Define set of candidates $I_c = \{1,\ldots,M\}\setminus I_{k-1}$
    \State Find $j_{\max} = \arg\max_{j\in I_c} -[\nabla F(\vwei^{k-1})]_j$
    \State Set $I_k = I_{k-1}\cup\{j_{\max}\}$
\State $\begin{aligned}
\text{Solve } &  \min_{ \vwei=[\wei_1,\ldots,\wei_M]^T } F(\vwei) \\
& \text{s.t. }  \wei_i \ge 0 \, \text{ for } i \in I_k,\; \wei_j = 0 \text{ for }  j \notin I_k
\end{aligned}$
\EndFor
\State Compute residual $\sqrt{F(\vwei^{\com{M}})}$
\end{algorithmic}
\end{algorithm}

Even with the adoption of a greedy procedure, offline training based on \Cref{prob:unc-CrTrain} remains highly computationally demanding, as it operates directly on raw snapshot data. This data typically exhibits near-linear dependencies among equations and the equation number scales with $K\Nr$.

The linear dependencies can cause convergence issues for certain greedy algorithms. While such issues can be mitigated using standard compression techniques \cite{art:globConv-onTheFly-hyperred, art:humphry2025efficient}, the computational cost poses a more critical challenge, as noted by several authors and shown in the numerical section. In \cite{art:hyperreduction-fem-cubature}, compression of the training data is proposed as a remedy. However, since this compression is applied to the raw solution manifold matrix $\tilde{\Ap}$, the approach still scales with the full complexity of \Cref{prob:unc-CrTrain}. To make it feasible, partitioned \acro{SVD} --an inexact compression technique-- is suggested. Similarly, \cite{hernandez2024cecm} advocates the use of randomized \acro{SVD}~\cite{art:MartinssonVoronin2016}.
The paper \cite{art:humphry2025efficient} introduces a more advanced and efficient modification of \Cref{alg:OrthMatchPur}, which incorporates incremental updates and ensures stability with respect to rounding errors.
Other works rely on massive parallelization \cite{art:equad-eff-yano,art:grimberg2021mesh}, enabling the treatment of large-scale problems given sufficient computational resources. These approaches reduce the cost of processing the $K\Nr \times M$ training matrix -- by distributing work or lowering the cost per operation -- but do not change its fundamental dimension.

In contrast, we pursue a different and complementary strategy: reducing the dimension of the training problem itself through structured data compression, so that it scales with $K$ alone rather than $K\Nr$. Our approach can in principle be combined with parallelization or other large-scale techniques mentioned above. However, the focus of this paper is exclusively on this preprocessing strategy.

\section{Preprocessing by structured data compression} \label{sec:preprocessing-ansatz}
The basic idea of our preprocessing approach is to approximate the solution manifold matrix $\tilde{\Ap}$ in \eqref{eq:Ab-data}.
Note that $\tilde{\Ap}$ consists of raw snapshot data, and its row dimension scales with $\Nr K$, i.e., with the \acro{ROM} dimension.
Since the matrix is dense, it can become very large in practical settings, potentially too large to fit into memory.

To address this and reduce the associated computational overhead, we propose preprocessing the training data through compression, resulting in a data-compressed training problem.
More specifically, we construct a structured low-rank approximation of $\tilde{\Ap}$ based on a suitable factorization of the data.
The compression is applied exclusively along the snapshot dimension $K$.
The advantage of this structured approach is that no step scales with $K$, $\Nr$, and $M$ simultaneously, in contrast to operating directly on $\tilde{\Ap} \in \R^{K\Nr \times M}$.
The structured compression problem underlying our approach is outlined in Section~\ref{subsec:preprocessing-ansatz}. Further required technical results and an abstract algorithmic implementation of the proposed preprocessing itself are provided in Section~\ref{subsec:solve-preprocess}.

\subsection{The preprocessing problem}\label{subsec:preprocessing-ansatz}
When $K$ is compressed to $\thi{K} \ll K$ in our approach, this yields a compressed solution manifold matrix $\thi{\Ap}$ and a respective compressed cost function of the form
%This relates to replacing the cost functional $\fopt(\vwei)$ in \Cref{prob:unc-CrTrain} by a compressed version, given by
\begin{align} \label{eq:foptThin}
	\thi{\fopt}(\vwei) = \| \thi{\Ap}( \vwei - \tilde{\vwei}) \|, \hspace{1cm} \text{with } \thi{\Ap}\in \R^{\thi{K}\Nr , M}.
\end{align}
The compressed training problem is then obtained replacing the exact cost function in \Cref{prob:unc-CrTrain} by its compressed counterpart.
\begin{prblm} \label{prob:con-CrTrain}
%%Let the  compressed function $\thi{\fopt}$ be defined as above for a given matrix~$\prol{\Ap}$.
Let $\thi{\fopt}$ be as in \eqref{eq:foptThin} and let $\tilde{\vwei}$ denote the truth weights.
\begin{align*} 
\text{Solve \quad}	 &\min_{ \substack{I_c\subset \{1,\ldots M\}, \, |I_c | \leq \com{M} \\  \vwei=[\wei_1,\ldots,\wei_M]^T } }  \hspace{0.2cm} \thi{\fopt}(\vwei)^2 + (\dreg^T( \vwei -  \tilde{\vwei}))^2
 \\[-0.1cm]
	& \hspace{0.3cm} {\small \smash{ \text{s.t.}  \quad \wei_m \geq 0 \, \text{ for } m \in I_c, \quad \wei_{\bar{m}}  = 0 \text{ for } \bar{m}  \notin I_c  }},
\end{align*}
\end{prblm}
In what follows, we motivate and derive the compression step itself.
We start by noting that $\tilde{\Ap} \in \mathbb{R}^{K\Nr , M}$ is obtained from three-dimensional data and represents one matricization of a tensor in $\mathbb{R}^{K , \Nr , M}$. For structured analysis, however, a different matricization is preferable, defined as
{\small
\begin{align}
\begin{split}\label{eq:C-data}
\tilde{\Cp} &=
\begin{pmatrix}
\beta_{\rm *}^1\!\left(f(\bv{x}^1), \rom{\tL}^{\!1}\right) & \ldots & \beta_{\rm *}^1\!\left(f(\bv{x}^K), \rom{\tL}^{\!1}\right) \\
\beta_{\rm *}^2\!\left(f(\bv{x}^1), \rom{\tL}^{\!1}\right) & \ldots & \beta_{\rm *}^2\!\left(f(\bv{x}^K), \rom{\tL}^{\!1}\right) \\
\vdots & \ddots & \vdots \\
\beta_{\rm *}^M\!\left(f(\bv{x}^1), \rom{\tL}^{\!1}\right) & \ldots & \beta_{\rm *}^M\!\left(f(\bv{x}^K), \rom{\tL}^{\!1}\right) \\
\vdots & \ddots & \vdots \\
\beta_{\rm *}^M\!\left(f(\bv{x}^1), \rom{\tL}^{\!\Nr}\right) & \ldots & \beta_{\rm *}^M\!\left(f(\bv{x}^K), \rom{\tL}^{\!\Nr}\right)
\end{pmatrix} \\
&=
\begin{pmatrix}
\vAp^{1,1} & \ldots & \vAp^{K,1} \\
\vdots & \ddots & \vdots \\
\vAp^{1,\Nr} & \ldots & \vAp^{K,\Nr}
\end{pmatrix}.
\end{split}
\end{align}
}%%
By construction, $\tilde{\Cp}$ contains the same entries as $\tilde{\Ap}$. To fully exploit the underlying structure, we employ a factorization 
\begin{align*}
\tilde{\Cp} = {\bt{N}} \hat{\bt{G}}, \hspace{1cm} \bt{N} \in \R^{N_r M, M_J}, \quad \hat{\bt{G}} \in \R^{M_J,K}
\end{align*}
with $M_J \geq M$, and where $\bt{N}$ is a high-dimensional but sparse matrix encoding the constraints inherent in the data. As will be shown, the optimization can be reduced to a problem of the dimension of $\hat{\bt{G}}$, which is significantly lower than the dimension of $\tilde{\Ap}$. The derivation of the aforementioned factorization of $\tilde{\Cp}$ depends on the specific training problem; details are provided in Section~\ref{sec:separatedData}. Setting these details aside, the structured training problem of our proposed preprocessing can now be stated.

\begin{prblm}\label{prob:cpca-C}
Let $\tilde{\Cp}$ from \eqref{eq:C-data} be given in the structured form $\tilde{\Cp} = {\bt{N}} \hat{\bt{G}}$ with $\bt{N} \in \R^{\Nr M, M_J}$ of full column rank and $\hat{\bt{G}}\in \R^{M_J,K}$. Let $\thi{K} \ll rank(\tilde{\Cp})$.
\begin{align*} 
\text{Solve}\quad \min_{ \prol{\bt{G}}, \myrank{\prol{\bt{G}}}  \leq \thi{K}}  \left \| \tilde{\Cp} - \bt{N} \prol{\bt{G}} \right\|_{\rm{F}}.
%%%, \hspace{0.2cm} \text{where }  \bt{N} \in \R^{\Nr M, M_J}, \quad \hat{\bt{G}}\in \R^{M_J,K}.
\end{align*}
%%for given $\thi{K} \ll rank(\tilde{\Cp})$, and for $\| \cdot \|_{\rm{F}}$ given by the Frobenius norm.
\end{prblm}

\begin{rmrk}\label{rem:best-rank-k-equivalence}

Since the Frobenius norm of a matrix is invariant under any rearrangement of its entries, the following two best rank-$k$ approximation problems minimize the same cost function:
\begin{align*}
\min_{\prol{\Cp},\, \myrank{\prol{\Cp}} \leq \thi{K}} \left\| \tilde{\Cp} - \prol{\Cp} \right\|_{\rm F}, 
\qquad \text{and} \qquad
\min_{\prol{\Ap},\, \myrank{\prol{\Ap}} \leq \thi{K}} \left\| \tilde{\Ap} - \prol{\Ap} \right\|_{\rm F},
\end{align*}
given $\tilde{\Ap}$ and $\tilde{\Cp}$ as in \eqref{eq:Ab-data} and \eqref{eq:C-data}, respectively.
Similarly, \Cref{prob:cpca-C} shares the same cost function, but its admissible set is a subset of the best rank-$k$ approximation problem for $\tilde{\Cp}$, which is of lower dimension. This reduction in dimension is the key to the efficiency of our approach.
\end{rmrk}

%%%%
\subsection{Solving the compression problem} \label{subsec:solve-preprocess}
Structured problems of the form given in \Cref{prob:cpca-C} have been studied in the context of constrained principal component analysis \cite{art:Takane1991, art:Takane2001}. Among other results, connections to the standard best rank-$k$ approximation problem have been established. The following lemma presents one such connection. For completeness, we also provide a brief proof sketch.

\begin{lemma} \label{lem:svd-sol-cpca}
Let ${\bt{N}}$ have the QR factorization $\bt{N} = \bt{Q} \bt{R}$.
Then $\bt{G}_p^*$ solves \Cref{prob:cpca-C}, if and only if, $\mathcal{G}^* =  \bt{R} \bt{G}_p^*$ solves
\begin{align}\label{eq:lem:svd-sol-cpca}
\min_{ {\check{\mathcal{G}}}, \rm{rank}(\check{\mathcal{G}}) \leq \thi{K}}  \left \| \bt{R}  \hat{\bt{G}} - \check{\mathcal{G}} \right\|_{\rm{F}}
\end{align}
\end{lemma}
\begin{proof}
Using $\tilde{\Cp} = {\bt{N}} \hat{\bt{G}}$, and exploiting the orthogonality of $\bt{Q}$ together with the invariance of the Frobenius norm under orthogonal transformations, we obtain
\begin{align*}
 \left \| \tilde{\Cp} - \bt{N} \prol{\bt{G}} \right\|_{\rm{F}} =  \left \| \bt{Q} (  \bt{R} \hat{\bt{G}} -  \bt{R} \prol{\bt{G}} ) \right\|_{\rm{F}} = \left \|   \bt{R} \hat{\bt{G}} - {\bt{R}} \prol{\bt{G}}  \right\|_{\rm{F}} = \left \|   \bt{R} \hat{\bt{G}} - \check{\mathcal{G}}  \right\|_{\rm{F}}
\end{align*}
for any choice of $\check{\mathcal{G}} =  \bt{R} \prol{\bt{G}}$. In particular, this holds for the minimizers $\bt{G}_p^*$ and $\mathcal{G}^* = \bt{R} \bt{G}_p^*$ of the two minimization problems, from which the claim follows.
\end{proof}
The problem \eqref{eq:lem:svd-sol-cpca} is a standard best rank-$k$ approximation, which can be solved via singular value decomposition (\acro{SVD}), see~\cite{book:GolubVanLoan}. Consequently, together with \Cref{lem:svd-sol-cpca}, the following two corollaries can be deduced.
\begin{crllr}\label{cor:truncSVD}
Assume the conditions of \Cref{lem:svd-sol-cpca} hold. Let $\sigma_1 \ge \dots \ge \sigma_{\thi{K}} \ge \ldots \ge 0$ denote the singular values of $\bt{R}\bt{G}_p^*$ in descending order. Let, further, $\sigma_{\thi{K}} > \sigma_{\thi{K}+1}$. 

Then the solution to \eqref{eq:lem:svd-sol-cpca} is the truncated \acro{SVD} $\mathcal{G}^* = \bt{U}_1^L \boldsymbol{\Sigma}_1 \bt{U}_1^R$, where $\boldsymbol{\Sigma}_1 = \mydiag(\sigma_1,\dots,\sigma_{\thi{K}})$, and $\bt{U}_1^L \in \R^{M_J , \thi{K}}$, $\bt{U}_1^R \in \R^{\thi{K} , K}$ contain the first $\thi{K}$ left and right singular vectors of $\bt{R}\bt{G}_p^*$, respectively.
\end{crllr}
%%
%%%%
The following corollary shows that the results of the standard best rank-$k$ approximation extend naturally to our structured setting, provided the QR factorization of $\bt{N}$ is available. In particular, it yields a compression bound in terms of the truncated singular values $\sigma_i$, $i> \thi{K}$, of $\bt{R}\bt{G}_p^*$.
\begin{crllr}\label{cor:err-truncSVD}
Assume the conditions of \Cref{cor:truncSVD}. Then the solution of  \Cref{prob:cpca-C} is given by $\bt{G}^*_p = \bt{R}^{-1} \bt{U}_1^L \bold{\Sigma}_1 \bt{U}_1^R$, and it fulfills the compression bound
\begin{align*}
	\left \| \tilde{\Cp} - \bt{N} {\bt{G}}_p^* \right\|_{\rm{F}} = \kappa, \qquad \text{with } \kappa = \sqrt{\sum_{i> \thi{K}} \sigma_i^2 }.
\end{align*}
%%where the  are the singular values of the matrix $\bt{R}  \hat{\bt{G}}$ in \eqref{eq:lem:svd-sol-cpca}. 
\end{crllr}
Before proceeding, we would like to highlight two aspects required for the efficient implementation of the latter results. Firstly, the QR factorization of $\bt{N}$ and the inverse $\bt{R}^{-1}$ can be determined very efficiently in our setting, exploiting the specific structure of the training problem, cf.~Section~\ref{sec:separatedData}. 
Secondly, the low-rank approximations are identified with low-order matrices in an algorithmic implementation. For example, the solution $\bt{G}_p^*\in\R^{M_J,K}$ of the problem given as in \Cref{lem:svd-sol-cpca} is identified with the low-order matrix 
\begin{align}\label{eq:Gt-trunc}
\bt{G}_t^* =  \bt{R}^{-1} \bt{U}_1^L \bold{\Sigma}_1  \in\R^{M_J,\thi{K}}, \hspace{0.5cm} \text{since } \bt{G}_p^* = \bt{G}_t^*  \bt{U}_1^R,
\end{align}
choosing the basis $\bt{U}_1^R$ in the identification. Similarly $\bt{N} \bt{G}_p^*$, which is the low-rank approximation of $\tilde{\Cp}$ we solve for in \Cref{prob:cpca-C}, is identified with the low-order matrix $\thi{\Cp} := \bt{N} \bt{G}_t^* \in \R^{\Nr M, \thi{K}}$. The compressed solution manifold matrix $\thi{\Ap}$ defining \Cref{prob:con-CrTrain} is then obtained by an appropriate restructuring of $\thi{\Cp}$.
Our preprocessing procedure can now be summarized in \Cref{alg:CPCA}.

\begin{algorithm}[H]
\caption{Structured preprocessing for offline training}
\label{alg:CPCA}
\begin{algorithmic}[1]%%[1,\setlength{\algorithmicindent}{2em}]
\State \textbf{Input:} Structured snapshot data $\tilde{\Cp} = \bt{N}\hat{\bt{G}}$, compression dimension $\thi{K}$
\State \textbf{Output:} Compressed matrix $\thi{\Ap}\in \R^{\thi{K}\Nr,M}$, compression error $\kappa$
\State Determine\textsuperscript{*} QR factorization $\bt{N} = \bt{Q}\bt{R}$
\State\label{step-svd:alg:CPCA}Compute truncated SVD of $\bt{R} \hat{\bt{G}}$:
$\begin{aligned}
\mathcal{G}^* &= \bt{U}_1^L \bold{\Sigma}_1 \bt{U}_1^R,\; \hspace{0.2cm} (\text{with error } \kappa)
\end{aligned}$
\State Compute\textsuperscript{*} $\bt{G}^*_t = \bt{R}^{-1}\bt{U}_1^L$
\State Form $\thi{\Ap}$ from $\thi{\Cp} := \bt{N}\bt{G}_t$ by \blsRevB{matrix reshaping}
\end{algorithmic}
\vspace{-0.8\baselineskip}
\noindent\begin{minipage}{0.8\linewidth}
{\makeatletter
{\renewcommand{\thefootnote}{*}\footnotetext{\textsuperscript{*}exploiting the favorable matrix structure (\textit{$\bt{R}$ diagonal, $\bt{Q}$ sparse})}}
\makeatother}
\end{minipage}
\end{algorithm}

\begin{rmrk}\label{rem:comput-complex-cpca}
Since the matrix $\bt{N}$ and its QR factors exhibit a highly favorable structure (cf.\ next section), the computational cost of \Cref{alg:CPCA} is dominated by the \acro{SVD} in step~\ref{step-svd:alg:CPCA}, which is performed on a matrix of significantly lower dimension than $\tilde{\Ap}$. The memory-critical assembly of $\tilde{\Ap}$ (or $\tilde{\Cp}$) is entirely avoided, as we operate directly on the factorized form of $\tilde{\Cp}$.
\end{rmrk}

\section{Structured representation of the training data} \label{sec:separatedData}
It remains to derive a structure-revealing factorization of $\tilde{\Cp}$ as postulated in~\Cref{prob:cpca-C} and to analyze the structure of the resulting factors.
In Section~\ref{subsec:motivate-factori} we motivate the structural features inherent in the training data from a function-space point of view. The specific construction of the factorization is case-dependent. We first address the empirical quadrature case in Section~\ref{subsec:separatedData-quad}, followed by a more general setting in Section~\ref{subsec:separatedData-gen}, which also encompasses the empirical cubature case.
The framework also accommodates settings beyond the scalar-valued case of Section~\ref{subsec:intro-nonlinearity}. Section~\ref{subsec:separatedData-tensorval} illustrates this for the geometrically induced nonlinearities studied in \cite{art:Guo2024,art:Guo2024-secOrd}, and discusses how the compression strategy proposed there relates to our constrained structured compression.
\subsection{Structure induced by bilinear forms} \label{subsec:motivate-factori}
Recall from the definition of $\tilde{\Cp}$ in \eqref{eq:C-data} that the training data is composed of the terms $\beta_{\rm{*}}^m\left( f(\bv{x}^k) , \rom{\tL}\!^n \right)$ defined by localized bilinear forms
\begin{align*}
 \beta_{\rm{*}}^m: f(\funSpace{V}) \times \rom{\funSpace{W}} \rightarrow \R, \hspace{1cm} m = 1,\ldots , M.
\end{align*}
The nonlinear manifold $f(\funSpace{V})= \{y = f(x), x\in \funSpace{V} \}$ is induced by the nonlinearity $f$ we would like to approximate. Not the full domain is considered, but only nonlinear snapshots related to the state snapshots $x^k\in \funSpace{V}$, cf.~\eqref{eq:opt-cubquad}, i.e., the set
\begin{align*}
 S_f = \{y = f(x) \text{ for } x\in \{x^1, \ldots, x^K \} \, \} \subset f(\funSpace{V}).  %%\hspace{1cm} \text{for given solution snapshots } x^1, \ldots, x^K.
\end{align*}
Owing to the nonlinearity of the domain $f(\funSpace{V})$, no suitable representation of $S_f$ exists that would permit the use of standard data compression techniques. To address this difficulty, we adopt a data representation based on the action of the nonlinearity on the bilinear forms $\beta_{\!*}^m$, rather than on the elements of $S_f$ themselves. In this way, the analysis of the data structure is carried out from a dual-space perspective. More precisely, the factorization of the data matrix $\tilde{\Cp}$ assumed in \Cref{prob:cpca-C} is derived from respective factorizations of the local bilinear forms $\beta_{\!*}^m$.

\subsection{Data representation -- empirical quadrature case} \label{subsec:separatedData-quad}
Considering the training problem for the empirical quadrature, the bilinear form $\bform^{\rm{*}} = \bform^{\rm{quad}}$ specifies to
\begin{align*}
\bform^{\rm{*}}(f(\xrf),\rom{\tL}\!^n) = \sum_{m=1}^{M} \tilde{\wei}_m \beta_{\rm{*}}^m(f(\xrf),\rom{\tL}\!^n)= \sum_{m=1}^{M} \tilde{\wei}_m {f} (\xrf({\xp}_m))\rom{\tL}\!^n({\xp}_m) 
\end{align*}
using $M = M_{\rm{quad}}$. In that case, the local bilinear forms $\beta_{\rm{*}}^m$ are separable already, in the sense that they can be divided into the two linear forms $f(\xrf) \mapsto f(\xrf({\xp}_m))$ and $\rom{\tL}\!^n \mapsto \rom{\tL}\!^n({\xp}_m)$. This simplifies our construction. Defining
{\small
\begin{align*}
	\bv{g}^k = 
	\begin{bmatrix}
			{f}(\fs{x}^k({\xp}_1)) \\
			\vdots \\
			{f}(\fs{x}^k({\xp}_{M}))
	\end{bmatrix}
\hspace{0.2cm} \text{for } k=1,\ldots K, \hspace{0.5cm} \text{and} \hspace{0.5cm}
	\tAlg^n = 
	\begin{bmatrix}
			\rom{\tL}\!^n({\xp}_1) \\
			\vdots \\
			\rom{\tL}\!^n({\xp}_M)
	\end{bmatrix}
	\hspace{0.2cm} \text{for } n=1,\ldots \Nr,
\end{align*}
}%%
we can write the sub-vectors $\vAp^{k,n}$ of \eqref{eq:C-data} in the factorized form
\begin{align*}
	\vAp^{k,n} = \bv{g}^k \circ \tAlg^n  \hspace{1cm} \text{ for } k= 1,\ldots K, \hspace{0.2cm} n=1,\ldots, \Nr.
\end{align*}
Using the equality $\bv{g}^k \circ \tAlg^n = \mydiag(\tAlg^n) \bv{g}^k$, the structure-revealing factorization assumed in \Cref{prob:cpca-C} specializes in this case to
{\small
\begin{align}
\begin{split} \label{eq:factori-equad}
\tilde{\Cp} = \bt{N} \hat{\bt{G}}, 
\hspace{0.4cm} & \text{with } \hspace{0.4cm}  \bt{N} \in \R^{M \Nr,M}, \hspace{0.2cm}  \hat{\bt{G}} \in \R^{M,K},\\
& \text{given by }\hspace{0.1cm}  \bt{N} = 
\begin{pmatrix}
	\mydiag\left(\tAlg^1\right) \\
	\vdots \\
	\mydiag\left(\tAlg^{\Nr}\right)
\end{pmatrix}, \hspace{0.3cm}
\hat{\bt{G}} = 
\begin{pmatrix}
\bv{g}^1 & \ldots & \bv{g}^K  
\end{pmatrix} .
\end{split}
\end{align}
}%%
With reference to \Cref{rem:comput-complex-cpca}, we note the specific structure of ${\bt{N}}$ in \eqref{eq:factori-equad} and its implications for the computational complexity of \Cref{alg:CPCA}. The columns of $\bt{N}$ are orthogonal to each other. Thus a QR factorization of $\bt{N}$ can be constructed analytically, with $\bt{R}$ the diagonal matrix whose entries are the reciprocals of the column norms of $\bt{N}$.

\subsection{Data representation -- non-separable case} \label{subsec:separatedData-gen}
Next, we study the more general setting, in which the local bilinear forms $\beta_{\rm{*}}^m$ are not separable themselves. Our strategy to deal with that is to decompose the local forms further into terms which are separable and then pursue with similar arguments as in the empirical quadrature case. In general, there may exist different decomposition of the $\beta_{\rm{*}}^m$. Here, we suggest a construction based on the underlying \acro{FOM} basis and the assumption that each basis function has a small local support.

Recall from our general setting in Section~\ref{sec:motivSetting} that any \acro{ROM} basis function $\rom{\tL}\!^n$ is also an element of the \acro{FOM} space. Therefore, any $\rom{\tL}\!^n$ can be written as a linear combination $ \rom{\tL}\!^n= \sum_{i=1}^N \lambda_i^n \tL^i$ for $\tL^1, \ldots, \tL^N$ given as the \acro{FOM} basis. Further, due to the small local support assumption on the \acro{FOM} basis, there exists for every $m \in \{1, \ldots M \}$ an index set $J^m \subset \{1, \ldots N \}$ with small cardinality $|J^m | \ll M$ such that the linear forms $\beta_{\rm{*}}^m(\cdot, \tL^i): f(\funSpace{V}) \to \R$ vanishes for $i \notin J^m$.
\begin{rmrk}
The cardinalities $J^m$ are small for the typical model reduction settings. For instance, when the \acro{FOM} originates from a discretization with nodal finite elements, $J^m$ equals the number of nodes per cell. For example, linear finite elements on a simplex mesh in spatial dimension $d$ yield $|J^m| = d+1$, whereas piecewise constant ansatz functions lead to $|J^m| = 1$.
\end{rmrk}
Under the posed assumptions, we may write
\begin{align*}
	\beta_{\rm{*}}^m(\cdot, \rom{\tL}^n) = \sum_{i= 1}^N \lambda_i^n \beta_{\rm{*}}^m(\cdot, \tL^i) = \sum_{i\in J^m} \lambda_i^n \beta_{\rm{*}}^m(\cdot, \tL^i) \hspace{1cm} \text{for } m = 1,\ldots, M.
\end{align*}
For any $x^k \in \funSpace{V}$ it then holds
\begin{align*}
	\beta_{\rm{*}}^m(f(x^k), \rom{\tL}^n) &=  \sum_{i\in J^m} \lambda_i^n \beta_{\rm{*}}^m(f(x^k), \tL^i) =  (\hat{\bv{p}}^{n,m})^T \hat{\bv{g}}^{k,m},
%%	 \\
%%	&\text{for } \hat{\bv{p}}^{n,m}= [\lambda^n_i]_{i\in J^m} \in \R^{|J^m|}  \text{ and } \hat{\bv{g}}^{k,m}= [\beta_{\rm{*}}^m(f(x^k), \tL^i)]_{i\in J^m} \in \R^{|J^m|}.
\end{align*}
for $\hat{\bv{p}}^{n,m}= [\lambda^n_i]_{i\in J^m} \in  \R^{|J^m|}$  and $\hat{\bv{g}}^{k,m}= [\beta_{\rm{*}}^m(f(x^k), \tL^i)]_{i\in J^m} \in  \R^{|J^m|}$. 
Define the sum of all cardinalities $M_J= \sum_{m=1}^M |J^m|$ and the concatenated vectors $\hat{\bv{p}}^{n}, \hat{\bv{g}}^{k} \in \R^{M_J}$ by
{\small
\begin{align*}
\hat{\bv{p}}^{n}=  \begin{bmatrix}
\hat{\bv{p}}^{n,1} \\ \vdots \\ \hat{\bv{p}}^{n,M}
\end{bmatrix}, \quad n = 1,\ldots \Nr \hspace{0.5cm} \text{and} \hspace{0.5cm}
\hat{\bv{g}}^{k}=  \begin{bmatrix}
\hat{\bv{g}}^{k,1} \\ \vdots \\ \hat{\bv{g}}^{k,M}
\end{bmatrix},  \quad k = 1,\ldots K.
\end{align*}
}%endsmall
Let $\bt{T} \in \R^{M_J, M}$ be the block-diagonal matrix that has $|J^m|$ one-entries on the $m$-th row. The sub-vectors $\vAp^{k,n}$ of \eqref{eq:C-data} can be written in the factorized form
\begin{align*}
	\vAp^{k,n} = \beta_{\rm{*}}(f(x^k), \rom{\tL}\!^n) = \bt{T} \left( \hat{\bv{p}}^{n} \circ  \hat{\bv{g}}^{k} \right) =  \left( \bt{T} \mydiag(\hat{\bv{p}}^{n}) \right) \hat{\bv{g}}^{k}.
\end{align*}
Summarizing, the structure-revealing factorization assumed in \Cref{prob:cpca-C} specializes in this case to
{\small
\begin{align}
\begin{split} \label{eq:factori-egenarl}
\tilde{\Cp} = {\bt{N}} \hat{\bt{G}}, 
\hspace{0.4cm} & \text{with } \hspace{0.4cm}  {\bt{N}} \in \R^{M \Nr,M_J}, \hspace{0.2cm}  \hat{\bt{G}} \in \R^{M_J,K},\\
& \text{given by }\hspace{0.1cm}  {\bt{N}} = 
\begin{pmatrix}
	\bt{T} \mydiag\left(\hat\tAlg^1\right) \\
	\vdots \\
	\bt{T} \mydiag\left(\hat\tAlg^{\Nr}\right)
\end{pmatrix}, \hspace{0.3cm}
\hat{\bt{G}} = 
\begin{pmatrix}
\hat{\bv{g}}^1 & \ldots & \hat{\bv{g}}^K  
\end{pmatrix} .
\end{split}
\end{align}
}%%endsmall
Note that the matrix ${\bt{N}}$ in \eqref{eq:factori-egenarl} separates into $M$ blocks of column vectors, where each block spans a space that is orthogonal to each of the other ones. Exploiting that structure and since $\bt{N}$ is sparse, it is possible to parallelize the calculation of the QR factorization efficiently, and the resulting $\bt{R}$ then also has a favorable sparse block-diagonal structure.

As an alternative to the previous approach, we propose and employ a simplified compression strategy. Rather than applying data compression to the full matrix $\tilde{\Cp}$ in \eqref{eq:factori-egenarl}, we restrict it to the factor $\breve{\Cp}$, defined as
{\small
\begin{align} \label{eq:factori-simple-egenarl}
\breve{\Cp} := \breve{\bt{N}} \hat{\bt{G}}, \hspace{0.5cm} \text{where } \breve{\bt{N}} = \begin{pmatrix}
	\mydiag\left(\hat\tAlg^1\right) \\
	\vdots \\
	\mydiag\left(\hat\tAlg^{\Nr}\right)
\end{pmatrix}.
\end{align}
}%%endsmall
This construction satisfies $\tilde{\Cp} = \bt{T}_{\rm all}\breve{\Cp}$, where $\bt{T}_{\rm all}$ is a block-diagonal matrix consisting of $\Nr$ copies of $\bt{T}$.
The QR factorization of $\breve{\bt{N}}$ involves only column-wise operations and produces a diagonal QR factor $\bt{R}$, closely mirroring the structure encountered in empirical quadrature settings and can therefore be implemented in a similar manner.

\begin{rmrk}
Within the simplified compression approach, we approximate $\tilde{\Cp} \approx \bt{T}_{\rm{all}} \prol{\breve{\Cp}}$, where $\prol{\breve{\Cp}}$ is the approximation obtained from using \Cref{alg:CPCA} to the data of $\breve{\Cp}$ (modulo a prolongation). Since $\|\bt{T}_{\rm{all}}\| =\|\bt{T}\| = \max_{m} \sqrt{J_m}$ due to the block-structure, the approximation can be bounded as
{\small
\begin{align*}
\| \tilde{\Cp} - \bt{T}_{\rm{all}} \prol{\breve{\Cp}} \|_{\rm{F}} = \| \bt{T}_{\rm{all}} ({\breve{\Cp}} - \prol{\breve{\Cp}}) \|_{\rm{F}} \leq 
 \|\bt{T}_{\rm{all}}\| \, \|\breve{\Cp} - \prol{\breve{\Cp}} \|_{\rm{F}} = \left(\max_{m} \sqrt{J_m} \right)\, \|\breve{\Cp} - \prol{\breve{\Cp}} \|_{\rm{F}},
\end{align*}
}%%endsmall
where we used the fact that  $\| \bt{T}_{\rm{all}} \bt{M} \|_{\rm{F}} \leq 
 \|\bt{T}_{\rm{all}}\| \, \| \bt{M} \|_{\rm{F}}$ holds for any matrix $\bt{M}$ independently of the specific structure of $\bt{T}_{\rm{all}}$.
\end{rmrk}

\subsection{Data representation -- a continuum mechanics case}
\label{subsec:separatedData-tensorval}

This subsection illustrates the extension of our framework to an example where test functions enter through their gradients and are accommodated with a parametric geometric transformation. The idea of treating such parametric geometric transformations within structured empirical quadrature is due to \cite{art:Guo2024,art:Guo2024-secOrd}.

Let $\phiVec^1,\ldots,\phiVec^{\Nr}: \Omega \to \R^d$ be $\R^d$-valued \acro{ROM} test functions, $d \in \{2,3\}$, on a fixed parent domain~$\Omega$.  Motivated by the setting of \cite{art:Guo2024}, we consider a parameter-dependent nonlinear field $\tilde{\nlField}$, with $\tilde{\nlField}(\xrfVec,\xi ;\boldsymbol{\mu}) \in \R^{d\times d}$ for a geometric parameter $\boldsymbol{\mu}$. The parameter dependence within $\tilde{\nlField}$ is not further detailed here. The reduced nonlinearity is defined as
\begin{align*}
[\rom{\ffun}(\xrfVec,\xi;\boldsymbol{\mu})]_n = \sum_{m=1}^M \tilde{\beta}_{\rm{*}}^m(\tilde{\nlField}(\xrfVec,\xi;\boldsymbol{\mu}),\nabla  \phiVec^n; \boldsymbol{\mu})
\end{align*}
with local bilinear forms
\begin{align*}
\tilde{\beta}_{\rm{*}}^m(\tilde{\nlField}(\xrfVec,\xi;\boldsymbol{\mu}),\nabla \phiVec^n;\boldsymbol{\mu})
:= \bigl(\nabla\phiVec^n(\xp_m)\,\geoTrans(\xp_m;\boldsymbol{\mu})\bigr)
\,\boldsymbol{\colon} \tilde{\nlField}(\xrfVec(\xp_m),\xp_m;\boldsymbol{\mu})
\end{align*}
where $\nabla\phiVec^n  \boldsymbol{\colon} \tilde{\nlField}= \sum_{i,j}[\tilde{\nlField}]_{ij}[\nabla\phiVec^n]_{ij}$ denotes the Frobenius inner product. The structural feature we like to highlight is that the invertible transformation $\geoTrans(\xp_m;\boldsymbol{\mu})\in\R^{d\times d}$ acts on the test function gradient, causing $\tilde{\beta}_{\rm{*}}^m$ to inherit a parameter-dependence. In that sense, this form is \emph{not} separable.
Following \cite{art:Guo2024}, the transformation is isolated to the data side by defining
$\nlField(\xrfVec,\xi;\boldsymbol{\mu}) := \tilde{\nlField}(\xrfVec,\xi;\boldsymbol{\mu})\,\geoTrans^T(\xi;\boldsymbol{\mu})$,
yielding the equivalent, parameter-separated local form
\begin{align*}
\beta_{\rm{*}}^m(\nlField(\xrfVec,\xi;\boldsymbol{\mu}),\nabla \phiVec^n)
:= \nabla\phiVec^n(\xp_m) \,\boldsymbol{\colon} \nlField(\xrfVec(\xp_m),\xp_m;\boldsymbol{\mu}).
\end{align*}
Now $\boldsymbol{\mu}$ enters exclusively through~$\nlField$, and $\nabla\phiVec^n$ is parameter-independent.

Via columnwise vectorization $\mathrm{vec}(\cdot): \R^{d\times d}\to\R^{d^2}$,
$\beta_{\rm{*}}^m$ is separable. We can write
$\beta_{\rm{*}}^m(\nlField(\xrfVec\!^k,\xi;\boldsymbol{\mu}^k), \phiVec^n) = \bigl(\hat{\bv{p}}^{n,m}\bigr)^T \hat{\bv{g}}^{k,m}$ with
\begin{align*}
\hat{\bv{p}}^{n,m} = \mathrm{vec}\!\left(\nabla\phiVec^n(\xp_m)\right) \in \R^{d^2}, \qquad
\hat{\bv{g}}^{k,m} = \mathrm{vec}\!\left(\nlField(\xrfVec\!^k(\xp_m),\xp_m;\boldsymbol{\mu}^k)\right) \in \R^{d^2},
\end{align*}
for snapshot pairs $(\xrfVec\!^k,\boldsymbol{\mu}^k)$, $k=1,\ldots,K$.
Defining $M_J = Md^2$, the factorization of \Cref{prob:cpca-C} specializes to $\tilde{\Cp} = {\bt{N}} \hat{\bt{G}}$ with
{\small
\begin{align*}
{\bt{N}} = \begin{pmatrix} \mydiag(\hat{\tAlg}^1) \\ \vdots \\ \mydiag(\hat{\tAlg}^{\Nr})
\end{pmatrix} \in \R^{\Nr M_J,\, M_J}, \qquad
\hat{\bt{G}} = \begin{pmatrix} \hat{\bv{g}}^1 & \cdots & \hat{\bv{g}}^K \end{pmatrix}
\in \R^{M_J,\, K},
\end{align*}
}%
with $\hat{\tAlg}^n = [(\hat{\bv{p}}^{n,1})^T,\ldots,(\hat{\bv{p}}^{n,M})^T]^T \in \R^{M_J}$, and $\hat{\bv{g}}^k \in \R^{M_J}$ defined accordingly.
As in Section~\ref{subsec:separatedData-quad}, the columns of ${\bt{N}}$ are orthogonal
and the QR factor $\bt{R}$ is diagonal.

\begin{rmrk}\label{rem:param-dep-trick}
The above rewriting generalizes beyond this specific geometric setting: the parameter dependence can be removed from the bilinear form,
$\tilde{\beta}_{\rm{*}}^m(\tilde{\nlField}(\cdot;\boldsymbol{\mu}), \phiVec^n;\boldsymbol{\mu}) = \beta_{\rm{*}}^m(\nlField(\cdot;\boldsymbol{\mu}), \phiVec^n)$,
whenever the dependence in the test functions enters through a common invertible factor~$\geoTrans$, the same for every basis function,
which may be parameter- or state-dependent.
\end{rmrk}

\begin{rmrk}
In the setting of \cite{art:Guo2024,art:Guo2024-secOrd}, $\geoTrans=\bt{F}_{\boldsymbol{\mu}}^{-1}$, and ${\nlField}$ can be used to determine the effective stress, which is a physical quantity of interest. The latter motivates compressing snapshots of $\nlField$ prior to the complexity reduction of the nonlinearity $\rom{\ffun}$, which is precisely what has been done in \cite{art:Guo2024,art:Guo2024-secOrd}.

The compression in our framework differs by solving a data problem related to the nonlinearity $f_r$ to optimality, specifically regarding constraints associated with the \acro{ROM} test space. %% (see next section for an error analysis).
Moreover, it does not require a physically motivated separation of the nonlinear snapshot data, allowing for broader applicability, as argued in Sections~\ref{subsec:motivate-factori} and~\ref{subsec:separatedData-gen}.
\end{rmrk}
%%end{revA}

\section{Error analysis of proposed preprocessing} \label{sec:error-analysis}
In this section, we derive an effective a posteriori bound and an a priori bound that relate \Cref{prob:con-CrTrain} to \Cref{prob:unc-CrTrain}, thereby bounding the training error in our proposed preprocessing approach.
%%This bound depends on the approximation quality of the compression as well as on the relation $| {\vwei} \|_1 \approx | \tilde{\vwei} \|_1$.
%%
For notational convenience, we introduce a prolongated representation of the compressed cost functional~$\thi{\fopt}$. Instead of \eqref{eq:foptThin}, we use the equivalent form
\begin{align} \label{eq:costFunProl}
	\thi{\fopt}(\vwei) = \| \prol{\Ap}( \vwei - \tilde{\vwei}) \|, \hspace{1cm} \text{with } \prol{\Ap}\in \R^{{K}\Nr , M},
\end{align}
where $\prol{\Ap} \approx \tilde{\Ap}$ is chosen such that $\prol{\Ap}= \tilde{\bt{Q}}\thi{\Ap}$ for an orthogonal matrix $\tilde{\bt{Q}} \in \R^{K\Nr , \thi{K}\Nr}$.
Recall that the proposed preprocessing is based on a structured approximation of $\tilde{\Cp} \approx \bt{N} {\bt{G}}_p^*$, where $\tilde{\Cp}$ is a \blsRevB{reshaping} of $\tilde{\Ap}$. Due to the invariance of the Frobenius norm under matrix \blsRevB{reshaping}, we have $\| \tilde{\Ap} - \prol{\Ap} \|_{\rm{F}}$ $=\|  \tilde{\Cp} - \bt{N} {\bt{G}}_p^* \|_{\rm{F}}$, cf.\,~\Cref{rem:best-rank-k-equivalence}.
Thus, the compression error
\begin{align} \label{eq:kap-error}
\kappa = \| \tilde{\Ap} - \prol{\Ap} \|_{\rm{F}} = \|  \tilde{\Cp} - \bt{N} {\bt{G}}_p^* \|_{\rm{F}} = \sqrt{\sum_{i > \thi{K}} \sigma_i^2}
\end{align}
can be computed efficiently using \Cref{cor:err-truncSVD}, with $\bt{N}$, ${\bt{G}}_p^*$ and the constant $\kappa$ as defined in the corollary. The truncated singular values $\sigma_i$ for $i> \thi{K}$ are computed in the compression step of \Cref{alg:CPCA} without incurring additional computational cost, since they are directly produced by the \acro{SVD}.
%%As stated earlier, the matrix on which the \acro{SVD} is performed has a much lower dimension than $\tilde{\Ap}$, and $\kappa$ falls off as a biproduct of the \acro{SVD}.
%
\begin{theorem}\label{the:CompressedOptFun}
Consider \Cref{prob:unc-CrTrain} and \Cref{prob:con-CrTrain}. Let $\prol{\Ap}$ be defined as in~\eqref{eq:costFunProl}, and let $\kappa$ be as in~\eqref{eq:kap-error}. Define $d_{\rm min}$ as the smallest entry of $\dreg$, and assume $d_{\rm min} > 0$. Let $\vwei$ be a feasible point of \Cref{prob:con-CrTrain} such that $|\dreg^T(\vwei - \tilde{\vwei})| \leq \epsilon$ for some $\epsilon \geq 0$.

Then the cost function $\fopt : \vwei \mapsto \|\tilde{\Ap}(\vwei - \tilde{\vwei})\|$ of \Cref{prob:unc-CrTrain} satisfies the following two bounds:
\begin{align*}
\fopt(\vwei)   &\leq  \thi{\fopt}(\vwei) + \kappa  ||\vwei - \tilde{\vwei}||,  &&  \text{(a posteriori bound)} \\
\fopt(\vwei)  &\leq \thi{\fopt}(\vwei) + \kappa  \left( \frac{\sqrt{\com{M}}}{d_{\rm{min}}} ( \epsilon+ \dreg^T \tilde{\vwei})  + \|\tilde{\vwei}\| ) \right).
&& \text{(a priori bound)}
\end{align*}
\end{theorem}
\begin{proof}
By definition of the cost functions, it follows
{\small
\begin{align*}
\fopt(\vwei) &=|| \tilde{\Ap} (\vwei - \tilde{\vwei}) || \leq || (\tilde{\Ap} - \prol{\Ap}) (\vwei - \tilde{\vwei}) || + || \prol{\Ap} (\vwei - \tilde{\vwei}) ||  \\
&=  || (\tilde{\Ap} - \prol{\Ap}) (\vwei - \tilde{\vwei}) || + \thi{\fopt}(\vwei) \leq || (\tilde{\Ap} - \prol{\Ap})|| \, ||\vwei - \tilde{\vwei}||  + \thi{\fopt}(\vwei) \\
&\leq  || (\tilde{\Ap} - \prol{\Ap})||_{\rm{F}} \, ||\vwei - \tilde{\vwei}||  + \thi{\fopt}(\vwei)  = \kappa ||\vwei - \tilde{\vwei}|| + \thi{\fopt}(\vwei),
\end{align*}
}%%
using the triangle inequality, sub-multiplicativity, and the consistency of the Frobenius norm with the spectral norm. This proves the a posteriori bound.

To proof the a priori bound, it remains to bound the term $||\vwei - \tilde{\vwei}||$ independently of $\vwei$. First note that
due to the non-negativity of $\vwei$ and $\dreg$, respectively, the one-norm of $\vwei$ fulfills
\begin{align*}
||\vwei||_1 = \sum_{m\in I_c} |\wei_m|  =  \frac{1}{d_{\rm{min}}} \left(\sum_{m\in I_c} d_{\rm{min}} \wei_m \right)  \leq  \frac{1}{d_{\rm{min}}}(\dreg^T \vwei ).
\end{align*}
Since $\vwei$ has at most $\com{M}$ nonzero entries, it can be identified with a vector in $\R^{\com{M}}$, implying the norm relation $||\vwei|| \leq \sqrt{\com{M}} ||\vwei||_1$. Using the triangle inequality, we obtain
\begin{align*}
\| \vwei - \tilde{\vwei} \| &\leq  \| \vwei \| + \|\tilde{\vwei}\|
\leq \sqrt{\com{M}}  \| \vwei \|_1  + \|\tilde{\vwei}\| 
 \leq      \frac{\sqrt{\com{M}}}{d_{\rm{min}}}(\dreg^T \vwei ) + \|\tilde{\vwei}\| 
\\ &\leq \frac{\sqrt{\com{M}}}{d_{\rm{min}}} ( \epsilon+ \dreg^T \tilde{\vwei})  + \|\tilde{\vwei}\| .
\end{align*}
The a priori bound now follows from inserting the latter bound on $\| \vwei - \tilde{\vwei} \|$ into the a posteriori bound.
\end{proof}
The first bound in \Cref{the:CompressedOptFun} provides a rigorous and computationally efficient estimate of the training error $\fopt(\vwei)$ that only depends on the compression error  $\kappa$ calculated in \Cref{alg:CPCA} but not on the full solution manifold matrix~$\tilde{\Ap}$. Since $\fopt$ is the true training error, this bound serves as a computable certificate for the quality of the compressed training, without requiring the explicit construction of the uncompressed training problem.
The second bound offers an a priori estimate of the difference $\fopt(\vwei) - \thi{\fopt}(\vwei)$ that is independent of the argument $\vwei$. Concerning the constants in the latter we highlight the following: The regularization term in \Cref{prob:con-CrTrain} enforces $\epsilon$ to be small, cf.~the paragraph after \eqref{eq:opt-cubquad}. The quantities $\|\tilde{\vwei}\|$, $d_{\rm min}$, and $\dreg^T \tilde{\vwei}$ depend only on the \acro{FOM}; in particular, $\dreg^T \tilde{\vwei}$ equals the volume of $\Omega$ under the setting of Section~\ref{sec:ansatzEquadCub}.

%%%%

%%%%%%%%%%%%%%%%%%%%%%%%%%%%%%%%%
%%%%%%% Numerical example

%% Gas-pipe specific
\newcommand{\blsVertex}{{{\nu}}} 
 \newcommand{\Onepipe}{e}
\newcommand{\tfb}{w}

\newcommand{\xrea}{{\rho}}
\newcommand{\xDom}{{\Omega}}
\newcommand{\xSurf}{{\Gamma}}
\newcommand{\cTests}{{C}}

\newcommand{\SKFEM}{\acro{SKFEM}}

\section{Numerical performance study} \label{sec:numericPerformance}
The focus of our numerical studies is the comparison of empirical quadrature/ cubature-based complexity reduction without and with our proposed preprocessing via structured compression. The methods using compression are referred to as '\acro{compressed training}' and those without '\acro{standard training}'. The compression is carried out using \Cref{alg:CPCA}. Two representative nonlinear test problems are considered: one defined on a three-dimensional cube and another modeling a gas transport network; see Fig.~\ref{fig:spatialDomains} for the spatial domains.

The three-dimensional problem consists of a nonlinear reaction–diffusion equation. With sufficiently fine spatial meshes, this example can be scaled to high dimensions. It serves as the primary benchmark for demonstrating improvements in offline computational time and memory requirements.

The network problem models gas distribution in a pipeline network using cross-sectionally averaged nonlinear flow equations. We adopt the structure-preserving approximation from \cite{art:bls-snapBasedMorComp}, which requires the use of the cell-based empirical cubature. In this example, we also demonstrate the tightness of the a posteriori error bound introduced in Section~\ref{sec:error-analysis}.

For both problems, we employ model order reduction based on proper orthogonal decomposition (\acro{POD}). Time discretization is performed using the implicit Euler method, combined with a Newton method for solving the resulting algebraic systems. All computations were carried out on an Intel Xeon E-2136 processor with 6 cores and 48~GB of RAM using \acro{Python} version 3.11.4, respectively \acro{MATLAB} Version 9.14.0 (R2023a).
Additional problem-specific parameter choices are provided below.
Further details can also be obtained directly from the accompanying code provided with this article.

%%The code can be m

\begin{figure}[htb]
    \centering
    \begin{minipage}[t]{0.48\textwidth}
        \centering
        \includegraphics[width=0.7\textwidth]{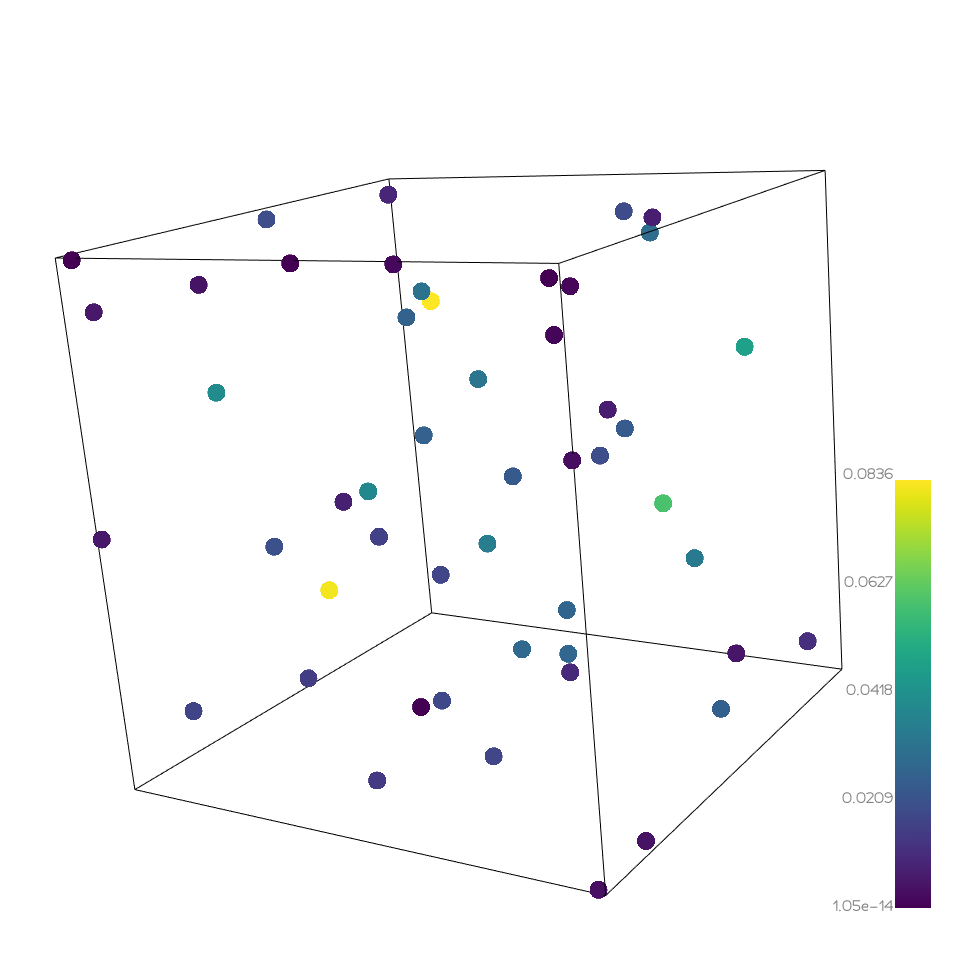}
    \end{minipage}
    \hfill
    \begin{minipage}[t]{0.48\textwidth}
        \centering
        \includegraphics[width=0.72\textwidth]{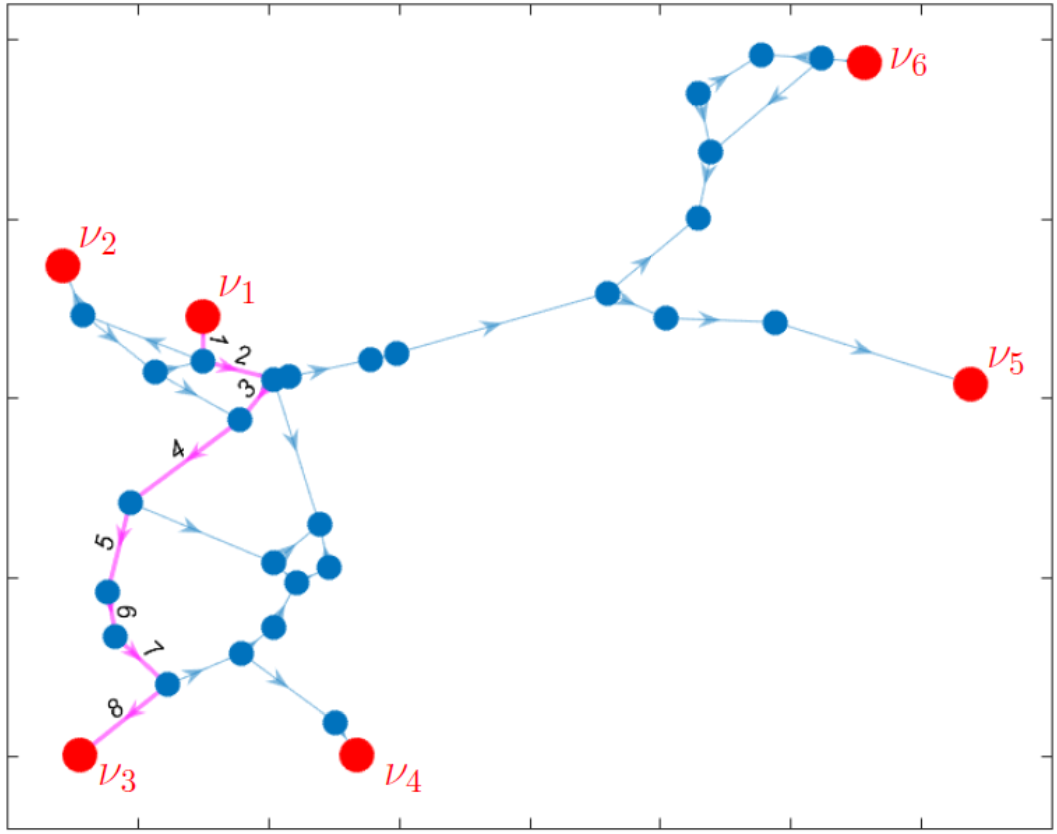}
    \end{minipage}

    \caption{Visualization of the spatial domains for both numerical benchmarks. Left: reaction-diffusion equation domain plus $M_c=50$ trained empirical quadrature points; right: gas network topology from~\cite{art:bls-snapBasedMorComp} with boundary nodes highlighted in red.}
    \label{fig:spatialDomains}
\end{figure}
%($M_c = 50$, $\thi{K} = M_c + 10$)

\subsection{Three-dimensional reaction-diffusion equation} \label{subsec:num-reacDiff}
Let $\xDom \subset \R^3$ denote the unit cube with boundary $\partial \Omega$ and outward normal vector $\bv{n}_{\xSurf}:\partial \Omega \to \R^3$.  
Furthermore, define the subset of the boundary in the upper-right corner by  $\xSurf_{\rm{top,right}}=  \{ [\xp_1,\xp_2,\xp_3] \in \partial \Omega \mid \xp_1 = 1 \text{ or } \xp_3 = 1 \}$, and the final time $t_{end}=1.5$
We consider the reaction-diffusion equation for $t\in (0,t_{end}]$,
\begin{align*}
\frac{\partial}{\partial_t} \xrea(t,\xp)  = -\nabla (\bt{D} \nabla \xrea(t,\xp) ) + f(\xrea(t,\xp)) \hspace{1cm} &\text{for} \, \xp \in \xDom \\
				(\bt{D} \nabla \xrea(t,\xp)) \cdot \bv{n}_{\xSurf}(\xp) =   g_{\xSurf}(t,\xp;C)  \hspace{1cm} &\text{for} \, \xp \text{ on } \xSurf_{\rm{top,right}} \\
				(\bt{D} \nabla \xrea(t,\xp)) \cdot \bv{n}_{\xSurf}(\xp) =   0  \hspace{1cm} &\text{for} \, \xp \text{ on } \partial \Omega \setminus \xSurf_{\rm{top,right}}\\
\end{align*}
The diffusion tensor is given by $\bt{D}= \mydiag([1, 0.5, 0.2])$ describing anisotropic diffusion. 
The nonlinearity is chosen as $f(\xrea) =  \xrea / (1 + 0.5\xrea)$, which is of the general form used to model saturating kinetics \cite{art:rd-johnson2011original,art:rd-Frank2018}.
{\small
\begin{figure}[htb]
  \centering

%%  % Row 1: z_slice = 0.1
%%  \begin{subfigure}[b]{0.32\textwidth}
%%    \includegraphics[width=0.9\textwidth]{figs/reaDif3_log_N30481_cBc0P75_fom/zSlice0P10_0P00.png}
%%    \caption{$t = 0.0$, $0\leq \xp_3 \leq 0.2$}
%%    \label{fig:z01_t0}
%%  \end{subfigure}
  \begin{subfigure}[b]{0.32\textwidth}
    \includegraphics[width=0.6\textwidth]{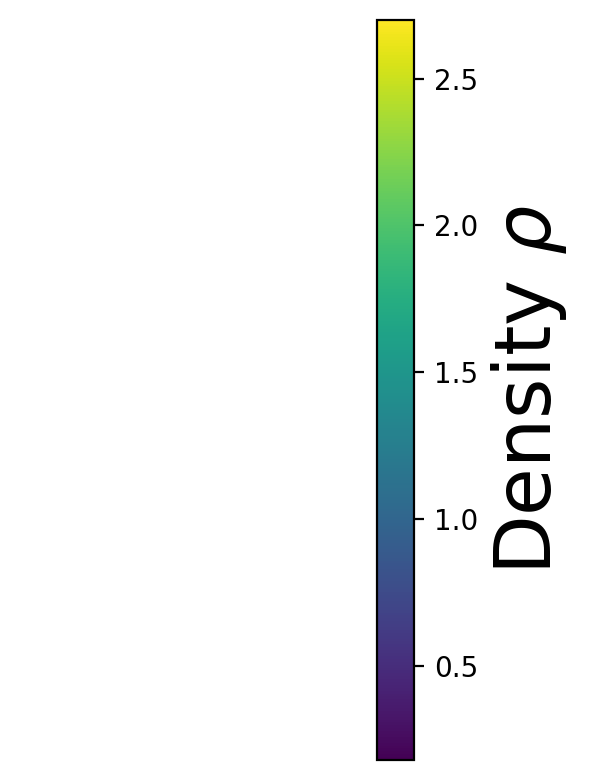}
    \caption{Common colorbar}
 %   \label{fig:z06_t15}
  \end{subfigure}  
  \hfill
  \begin{subfigure}[b]{0.32\textwidth}
    \includegraphics[width=0.9\textwidth]{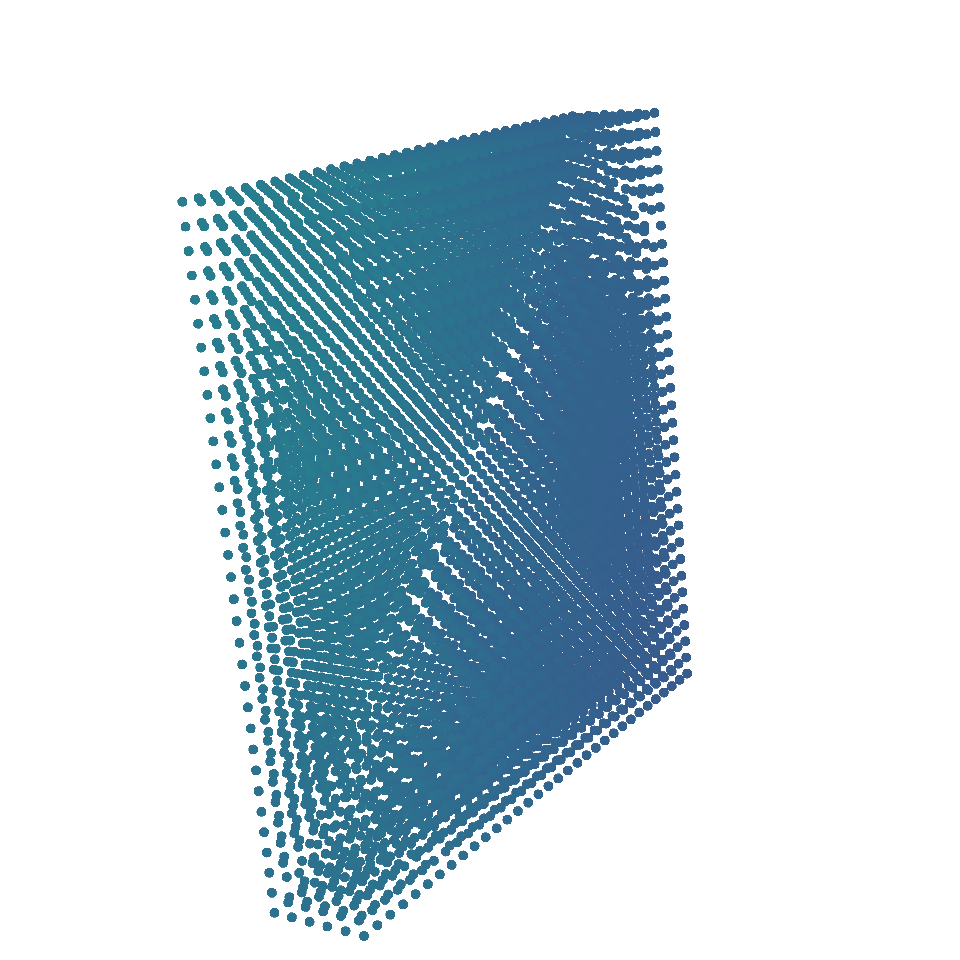}
    \caption{$t = 0.5$, $0\leq \xp_3 \leq 0.2$}
    \label{fig:z01_t05}
  \end{subfigure}
  \hfill
  \begin{subfigure}[b]{0.32\textwidth}
    \includegraphics[width=0.9\textwidth]{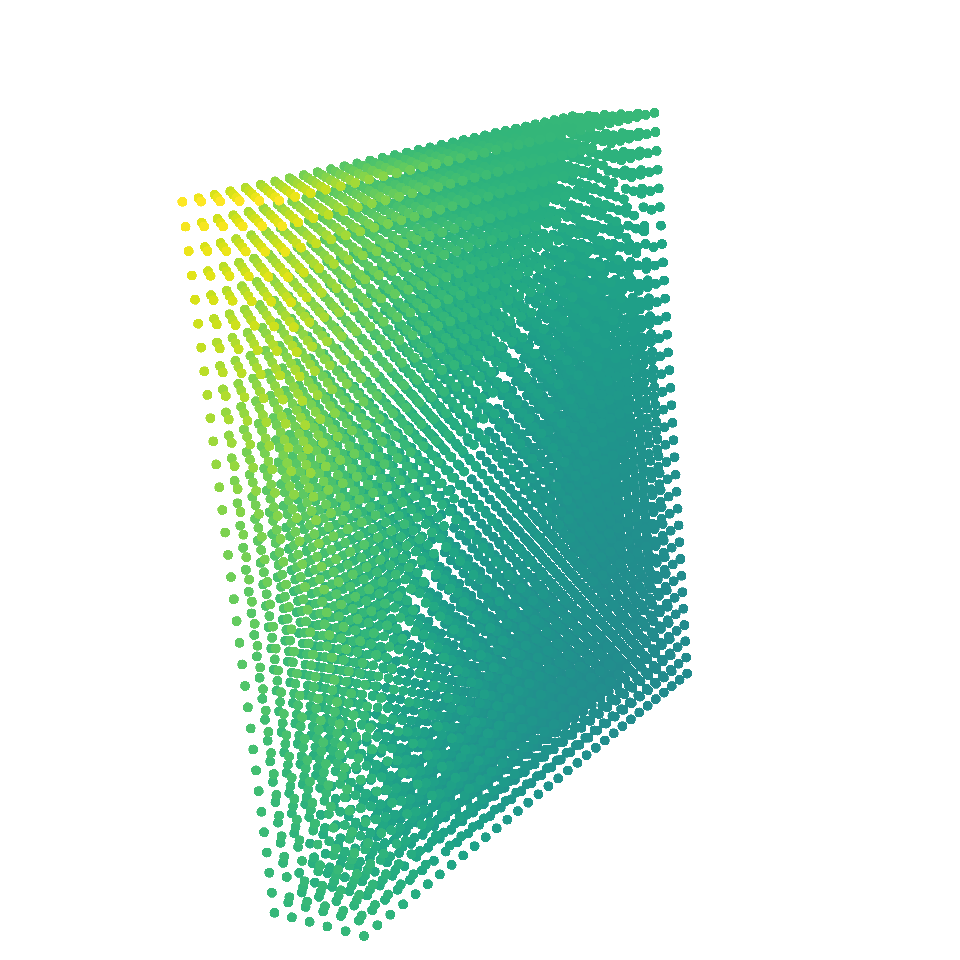}
    \caption{$t = 1.5$, $0\leq \xp_3 \leq 0.2$}
    \label{fig:z01_t15}
  \end{subfigure}

  \vspace{-1.2em}
  % Row 2: z_slice = 0.6
  \begin{subfigure}[b]{0.32\textwidth}
    \includegraphics[width=0.9\textwidth]{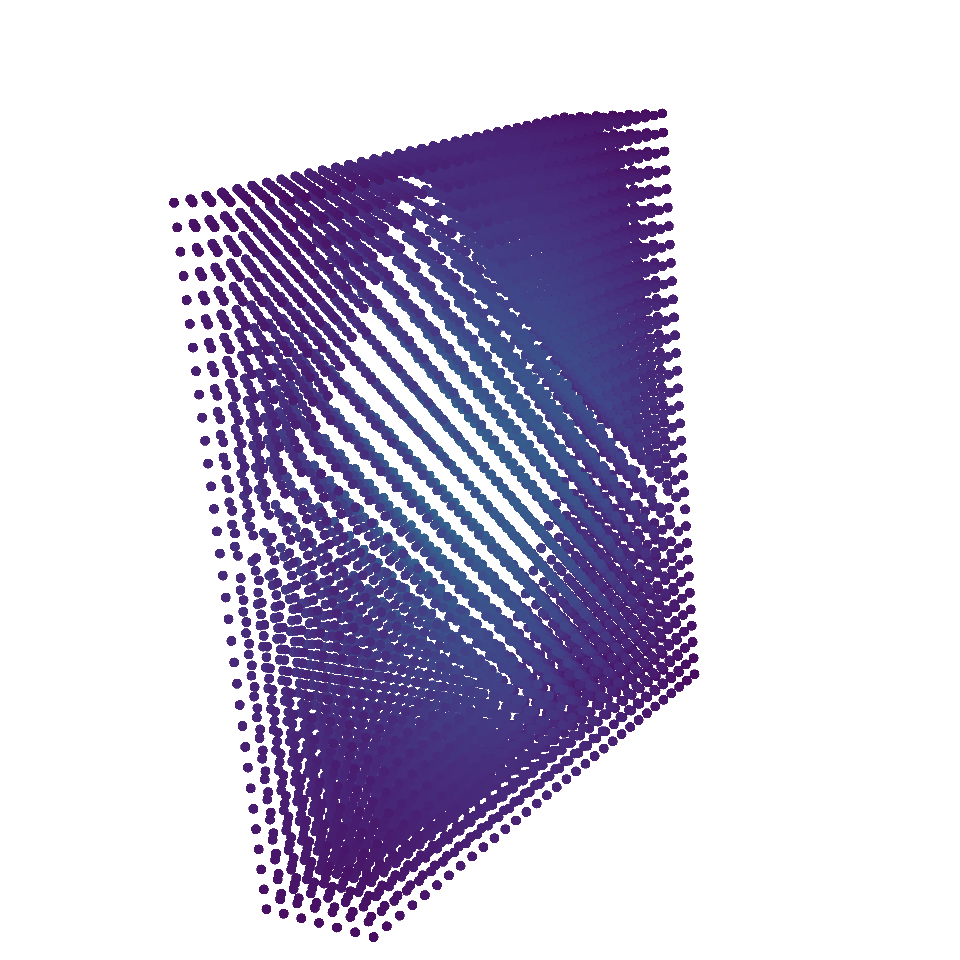}
    \caption{$t = 0.0$, $0.5\leq \xp_3 \leq 0.7$}
    \label{fig:z06_t0}
  \end{subfigure}
  \hfill
  \begin{subfigure}[b]{0.32\textwidth}
    \includegraphics[width=0.9\textwidth]{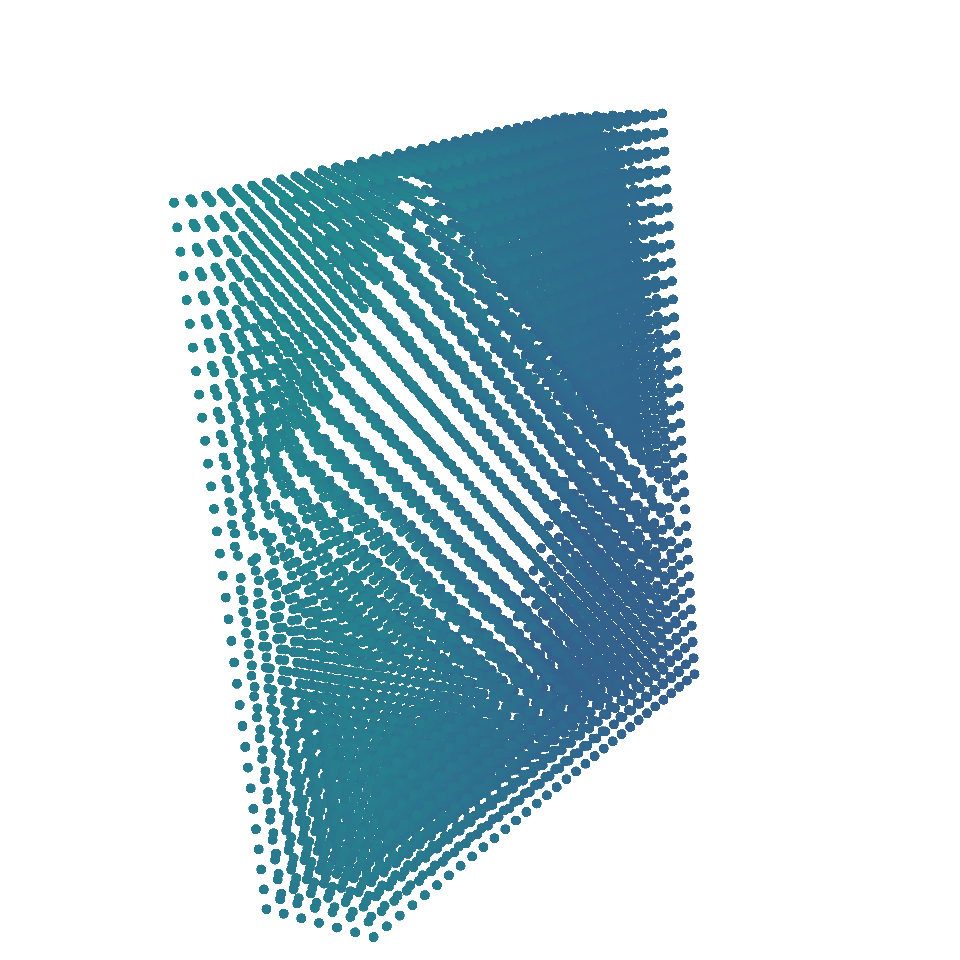}
    \caption{$t = 0.5$, $0.5\leq \xp_3 \leq 0.7$}
    \label{fig:z06_t05}
  \end{subfigure}
  \hfill
  \begin{subfigure}[b]{0.32\textwidth}
    \includegraphics[width=0.9\textwidth]{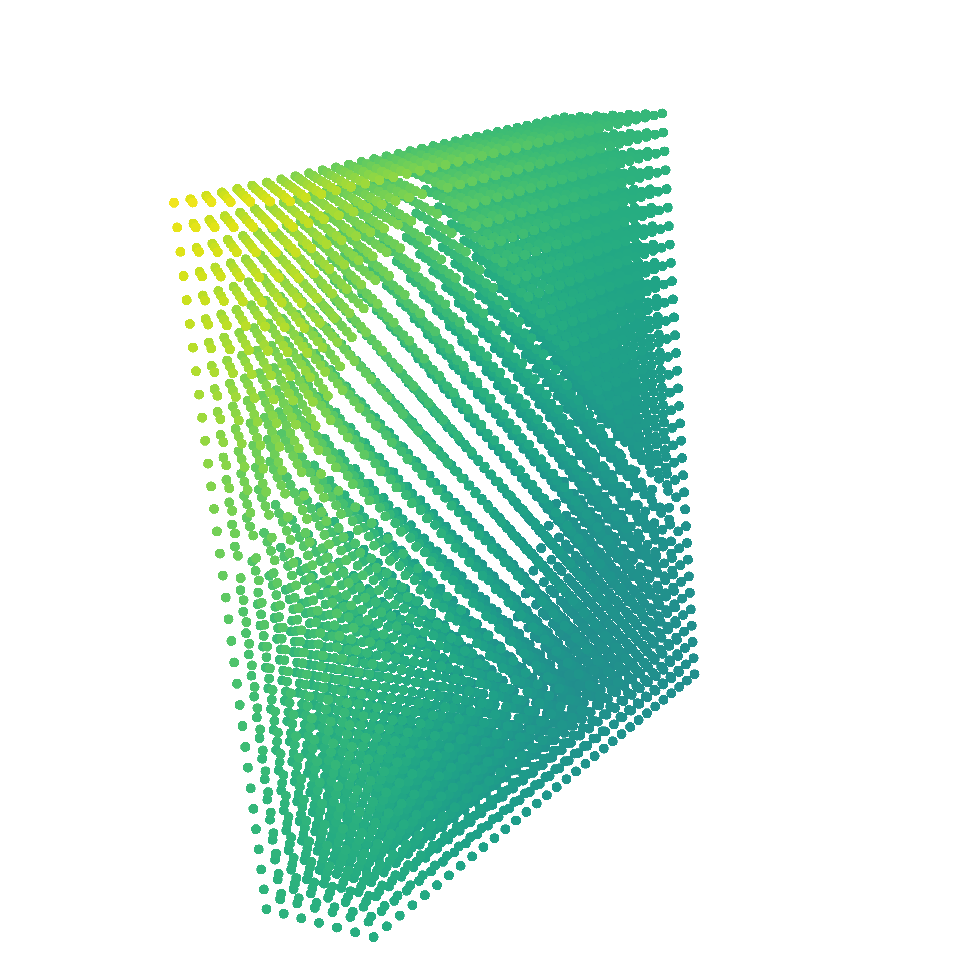}
    \caption{$t = 1.5$, $0.5\leq \xp_3 \leq 0.7$}
    \label{fig:z06_t15}
  \end{subfigure}
  \caption{Reaction-diffusion equation with $\cTests =0.75$: \acro{FOM} solution visualized by slices of width 0.2 at different time points. (\acro{FOM} with $163\,840$ cells.)}
  \label{fig:reacdiff_slices}
\end{figure}
}%%
The system is closed by initial conditions $\xrea(0,\xp) = \xrea_0(\xp;C)$ for $\xp \in \xDom$, where $\xrea_0$ is varied over a parameter $\cTests \in [0,1]$, according to
{\small
\begin{align*}
 \xrea_0(\xp;C) = (1-\cTests)   \exp \left( -\sum_{i=1}^3 \frac{(\xp_i - 0.5)^2)}{0.1} \right)  + \, \cTests \,\exp \left( -\sum_{i=1}^3 \frac{(\xp_i - 0.5)^2)}{0.5} \right)
\end{align*}
}
for  $\xp=[\xp_1,\xp_2,\xp_3] \in \xDom$. Similarly, the Neumann boundary conditions are varied according to
{\small
\begin{align*}
g_{\xSurf}(t,\xp;C) &= (1-\cTests)  \, g_{1}(\xp) \sin(6  t)(t) \, + \, \cTests \, g_{2}(\xp) (t - 0.2) \cos(4t)\\
\text{where }& g_{1}(\xp) = \sum_{i=1}^3 \xp_i, \hspace{0.4cm} \text{and} \hspace{0.3cm}
	g_{2}(\xp) = \sin(\xp_1) + \cos(6 \xp_2) (0.3-\xp_3^2).
\end{align*}
}
We refer to Fig.~\ref{fig:reacdiff_slices} for a visualization of the solution for the trajectory with $\cTests =0.75$.

The spatial discretization is carried out using the open-source Python library \acro{scikit{\-}-fem}\footnote{\url{https://github.com/kinnala/scikit-fem/releases/tag/11.0.0}} \cite{skfem2020}, employing linear finite elements on a tetrahedral mesh with varying levels of refinement. The nonlinear term is integrated numerically using a four-point symmetric Gauss rule, meaning that the number of quadrature points in the \acro{FOM} is four times the number of cells.

For all results, we use a \acro{ROM} dimension of $\Nr=35$. The training of the \acro{POD} basis and the empirical quadrature is carried out with snapshots from the training scenarios $\cTests \in \{0, 0.5, 1\}$. Using a constant time step of $\Delta_t = 0.002$ and recording snapshots every second time step ($2\Delta_t = 0.004$), we collect a total of $1128$ snapshots for both the state and the nonlinearity
%%By taking a snapshot every second time step, and since we employ a constant time step $\Delta_t = 0.002$, we obtain a total of $1128$%%$2253$
%%snapshots for both the state and the nonlinearity.

\paragraph{Performance of standard vs.\,compressed training}
As is shown in Fig.~\ref{fig:rd-numEq}, our preprocessing significantly reduces memory requirements without noticeably affecting the fidelity of the reduced model. For this example, we consider the overall online error of the resulting methods in the right part of the figure; it can be observed that the difference of the approximation error is several orders smaller than the approximation errors themselves.
The left part of the figure shows that the number of solution manifold equations for the respective methods. For the \acro{standard training} we get about $40{\,}000$ equations, i.e,. that amount of rows in the solution manifold matrix $\tilde{\Ap}$. The number of columns corresponds to the number of quadrature points and thus depends on the \acro{FOM}. For the finest discretization we have $1\,533{\,}520$ quadrature points, which results in roughly $3$ billion entries in the dense $\tilde{\Ap}$ and a memory requirement of $450$~GB (assuming 8 bytes per entry). Thus, this setting and others are not treatable with the standard approach on a normal working station.
In contrast, the \acro{compressed training} reduces the number of equations to about $5\%$, and thus the memory requirements are reduced accordingly, implying a reduction by a factor of $20$.

\begin{figure}[htb]
    \centering
    \begin{minipage}[t]{0.45\textwidth}
        \centering
        \includegraphics[width=0.88\textwidth]{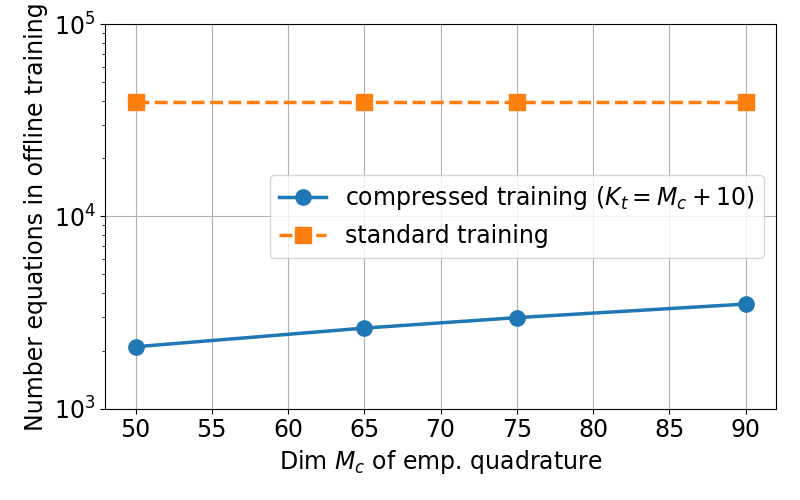}
    \end{minipage}
    \hfill
    \begin{minipage}[t]{0.45\textwidth}
        \centering
        \includegraphics[width=0.88\textwidth]{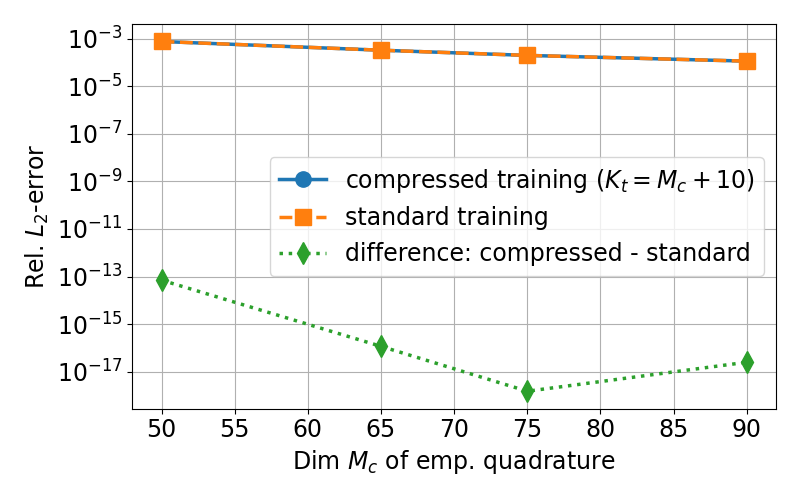}
    \end{minipage}
\caption{Reaction-diffusion equation. Training performance for varying $M_c$, with and without compression. Left: Number of equations in offline training (independent of \acro{FOM} dimension). Right: Relative  space-time $L_2$-errors of \acro{CROM} for the scenario with $\cTests =0.75$ and \acro{FOM} with $56{\,}025$  cells.
\label{fig:rd-numEq}}
\end{figure}

%%%%%%%%%%%%%%%%%%%%%%%%%%%%%%%%%%%%%%%%%%%%%
%%%%%%%% OFFLINE TIMES
\begin{figure}[htb]
    \centering
    \begin{subfigure}[b]{0.32\textwidth}
        \includegraphics[width=\textwidth]{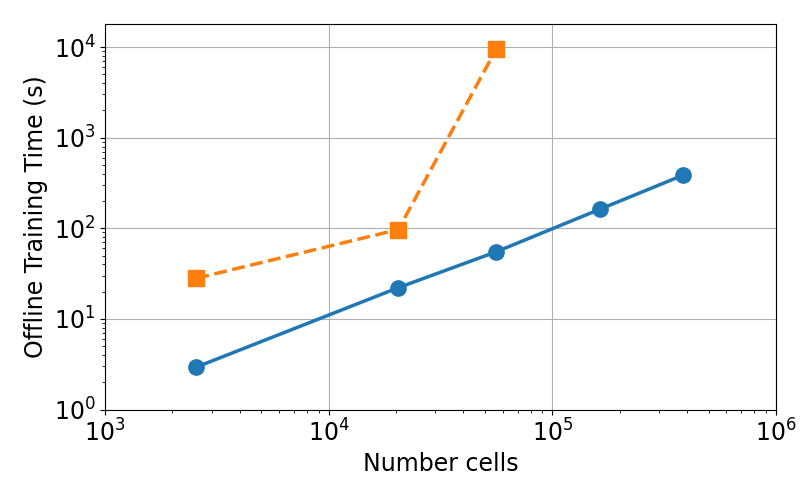}
        \caption{$M_c = 50$}
    \end{subfigure}
    \hfill
    \begin{subfigure}[b]{0.32\textwidth}
        \includegraphics[width=\textwidth]{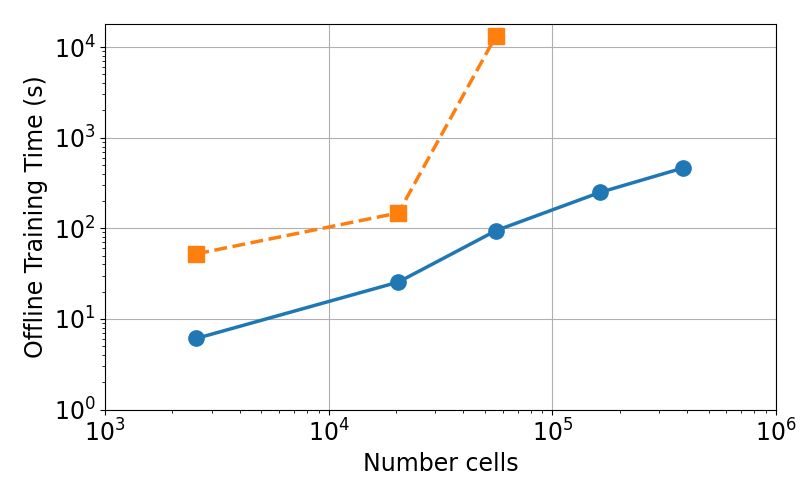}
        \caption{$M_c = 75$}
    \end{subfigure}
    \hfill
    \begin{subfigure}[b]{0.32\textwidth}
        \includegraphics[width=\textwidth]{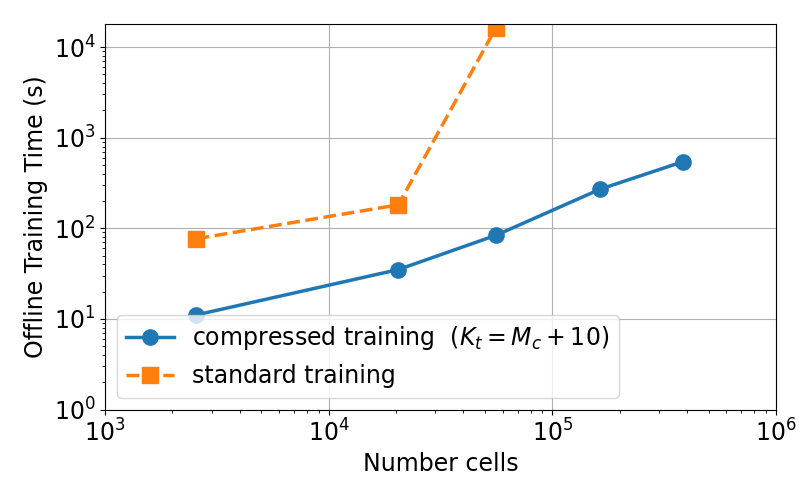}
        \caption{$M_c = 90$}
    \end{subfigure}
    \caption{Reaction-diffusion equation. Offline training times with and without compression, the varying number of cells in the \acro{FOM} and $M_c$. \label{fig:offline-training}}

\end{figure}

In Fig.~\ref{fig:offline-training}, we compare the offline training times for a varying number of cells in the \acro{FOM}. For the two lowest dimensional \acro{FOM}s, we see a gain of about one order. The third-lowest with about $56{\,}025$ cells is the largest, for which the \acro{standard compression} does not run out of memory on our machine, and there we see a gain of two orders by our compression, i.e., about $10\,000$ seconds against $100$ seconds training time.

\subsection{Gas transport network} \label{subsec:gas-Network}
%%%Para used  dx = 5.0e1; %%about 2e4 cells and 4e4 DOFs for FOM
%%%    thpara= solver.ThetaPara('rtol',1.0e-8, 'itmax',10, 'itmin',1, 'dt',1.5);
%% NR = 70, NR_2 = 40 (==70-29rho-1mu)
The second example is based on the gas distribution network model \cite{art:bls-snapBasedMorComp}, and it employs a mixed finite element discretization combined with a structure-preserving variant of \acro{POD} for the model order reduction. We provide only a brief outline of the essential model and highlight solely those aspects of the structure-preserving approximation that directly affect the complexity reduction, since all remaining details and parameter choices are documented in \cite{art:bls-snapBasedMorComp,code:bls22-phapprox} and can be assessed from the code accompanying the paper. %TODO

The network topology (Fig.~\ref{fig:spatialDomains}-right) is  adapted from the benchmark data set \cite[GasLib-40]{art:gaslib-2017}, where compressors have been replaced by standard nodes. It comprises 38 pipes with diameters $D^\Onepipe$ ranging from $0.4$ to $1~\mathrm{m}$ and lengths $l^\Onepipe$ adding up to to $1008~\mathrm{km}$ in total.
Each pipe is identified with an interval $[0,l^\Onepipe]$, and the spatial domain $\Omega$ is defined as their union. The density and mass flux $\rho,m: [0,T]\times \Omega \rightarrow \mathbb{R}$ satisfy the Euler-type equations
\begin{align} \label{bls-eq:abstr-a}
\partial_t \rho + \partial_\xp m &= 0, \qquad
\partial_t \frac{m}{\rho} + \partial_\xp \frac{m^2}{2\rho^2} + \partial_\xp P'(\rho) = -r(\rho,m)\,m,
\end{align}
where $P'$ denotes the derivative of a pressure potential modeling an isothermal pressure law, and the friction term is given by $r: (\rho,m) \mapsto 0.008 |m|/(2D^\Onepipe \rho^2)$. The individual pipes are interconnected through energy-conservative coupling conditions enforcing continuity of the specific stagnation enthalpy and mass conservation. The system is closed by one boundary condition per boundary node and by initial conditions.
%%(marked in red in Fig.~\ref{fig:spatialDomains}-right) 

%%
Apart from a refined spatial step size of $\Delta_\xi = 50~\mathrm{m}$, we adopt the same parameters as in the reference. %We report only those parameters directly affecting the complexity-reduction training problem.
This choice yields $20{\,}180$ cells and a \acro{FOM} dimension of $N = 40{\,}426$. The \acro{ROM} dimension is splitted according to $\Nr = N_{r,\rho} + N_{r,m} = 29 + 40$ on \acro{ROM} states relating to $\rho$ and $m$, respectively.
We use $K = 1200$ snapshots for each of the three nonlinear terms in \eqref{bls-eq:abstr-a}, which are uniformly distributed over the training trajectory described in~\cite{art:bls-snapBasedMorComp}.
More specifically, snapshots of the following mappings are collected:
{\small
\begin{align*} %%\label{eq:Sf}
	(\rho,m) \mapsto  \left( P'(\rho)+ \frac{m^2}{2 \rho^2} \right) \partial_\xp \tfb^j_r , \qquad 
	(\rho,m) \mapsto r(\rho,m) m \, \tfb^j_r  , \qquad  
	(\rho,m) \mapsto 	\frac{m}{\rho} \tfb^j_r,
\end{align*}
}for $\tfb^j_r$ with $j = 1,\ldots, N_{r,m}$, which are piecewise linear \acro{ROM} ansatz functions relating to $m$. Since the third nonlinearity appears as a time derivative, we approximate it using finite differences in time. Our structured preprocessing method is applied to the snapshots of the first mapping, as well as to the union of the snapshots of the second and third mappings, where the latter is handled according to the simplified approach in \eqref{eq:factori-simple-egenarl}.

\paragraph{Performance of standard vs.\,compressed training}
For this example, the use of the \acro{compressed training} yields a speedup of roughly a factor of four; for example, training with $M_c=120$ cubature weights requires about $35$ seconds with compression compared to approximately $140$ seconds without, see Fig.~\ref{fig:gas:times-err-bound}-left. It can also be observed that the \acro{compressed training} times depend less on $M_c$, since the compression itself accounts for most of the computational effort.

The additional training error introduced by compression is negligibly small, as can be seen by comparing the difference $\fopt - \thi{\fopt}$ with the residual of the training $\fopt$ itself in Fig.~\ref{fig:gas:times-err-bound}-right. Moreover, the figure shows the strict a posteriori bound on $\fopt$. This bound is not only computable in practice but also effective, as it results in an overestimation of only about one order of magnitude.

\begin{figure}[htb]
    \centering
%%\ifmakepreprint    
    \begin{minipage}[t]{0.45\textwidth}
        \centering
        \includegraphics[width=0.82\textwidth]{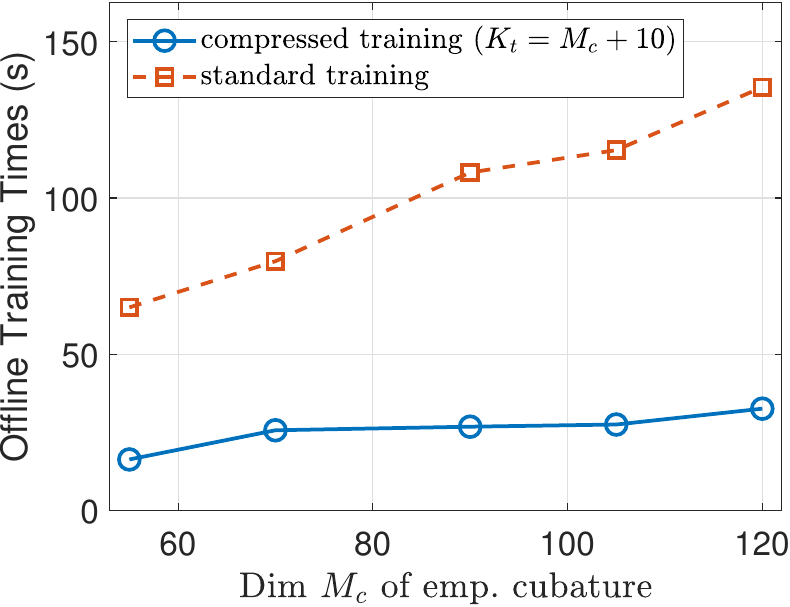}
    \end{minipage}
    \hfill
    \begin{minipage}[t]{0.45\textwidth}
        \centering
        \includegraphics[width=0.82\textwidth]{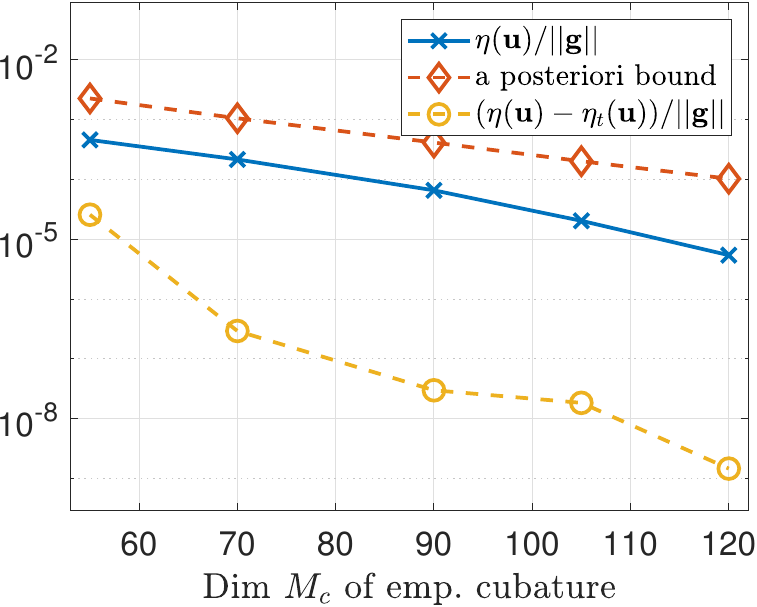}
    \end{minipage}
%%%%\else
%%    \begin{minipage}[t]{0.45\textwidth}
%%        \centering
%%        \includegraphics[width=0.82\textwidth]{figs/gasNet/ss-tcaseNgl38oscFF-traiOffArt-fi2_offT.eps}
%%    \end{minipage}
%%    \hfill
%%    \begin{minipage}[t]{0.45\textwidth}
%%        \centering
%%        \includegraphics[width=0.82\textwidth]{figs/gasNet/ss-tcaseNgl38oscFF-traiOffArt-fi2_bound.eps}
%%    \end{minipage}
%%\fi
\caption{Gas network. Comparison of training performance for varying $M_c$, with and without compression (using $\thi{K} = M_c+10$). Left: offline training times; right: training residuals, the difference introduced by compression, and the corresponding a posteriori bound from \Cref{the:CompressedOptFun} (scaled by $\|\bv{g}\|$, where $\bv{g}$ denotes the right-hand side of the training problem).}
    \label{fig:gas:times-err-bound}
\end{figure}

%%%%See readme file:
%%\begin{verbatim}
%%CodesFromOthers\CECM-continuous-empirical-cubature-method-main\README_files
%%# Example 4
%%## Sequential Randomized Singular Value Decompositions
%%To compare the performance of both the standard SVD and the proposed SRSVD,
%%launch script: [EXAMPLE_partSVD/TestSeqSVD.m](EXAMPLE_partSVD/TestSeqSVD.m)
%%\end{verbatim}
%%

\section*{Conclusion and outlook}
In this paper, we propose a preprocessing technique for empirical quadrature and other project-then-approximate complexity-reduction methods. The approach leverages structured data compression by exploiting the inherent structure of the training problems, thereby reducing memory requirements and offline training times by roughly an order of magnitude. This enables the application of such methods to larger-scale problems. The efficiency gains are demonstrated through numerical experiments, including tests that integrate our approach into the lightweight finite element library \acro{scikit-fem}. The paper deals with model reduction for standard finite element discretizations.

Potential directions for future work include extending the approach to other and more general discretization settings and exploring algorithmic aspects such as efficient on-the-fly updates of the complexity reduction. Furthermore, it would be interesting to assess the performance of our method in highly parallelized environments, which are essential for very large-scale problems.

\section*{Acknowledgements}
During the preparation of this work, the author used GitHub Copilot to refine the language in certain parts of the manuscript. After using this tool, the author edited the content as needed and takes full responsibility for the content.

\section*{Availability of data}
\blsRevB{The code used to produce the numerical results can be found under \cite{zenodo:comRedBilCode2026}.}

\bibliographystyle{abbrv}
\bibliography{ComRedBil25}

\end{document}